\documentclass[12pt]{amsart}
\usepackage[top=1in, bottom=1in, left=1in, right=1in]{geometry}

\usepackage{amssymb,amsmath,amsthm}
\usepackage{mathrsfs} 
\usepackage{bbm} 
\usepackage{enumitem}
\usepackage[hyphens]{url}
\usepackage{times}

\usepackage{graphicx}
\usepackage{color}
\usepackage[colorlinks=true]{hyperref}   

\numberwithin{equation}{section}

\newtheorem{thm}{Theorem}[section]
\newtheorem{lem}[thm]{Lemma}

\newtheorem{conjecture}[thm]{Conjecture}

\theoremstyle{definition}

\theoremstyle{remark}
\newtheorem{remark}{Remark}[section]
\newtheorem*{remark*}{Remark}

\newtheoremstyle{claim} 
    {1em}                    
    {1em}                    
    {}                   
    {}                           
    {\bfseries}                   
    {.}                          
    {.5em}                       
    {}  
    
\theoremstyle{claim}

\newcommand{\bs}\boldsymbol{}

\newcommand{\sumh}[0]{\sideset{}{^h}\sum}

\DeclareMathOperator{\meas}{meas}

\DeclareMathOperator{\sym}{sym}
\DeclareMathOperator{\SL}{SL}

\newcommand{\floor}[1]{\left\lfloor#1\right\rfloor}

\newcommand{\brac}[1]{\langle#1\rangle}

\renewcommand{\tilde}{\widetilde}

\renewcommand{\phi}{\varphi}

\renewcommand{\Re}{{\rm Re}}

\renewcommand{\bar}[1]{\overline{#1}}

\definecolor{red}{rgb}{1,0,0}

\definecolor{orange}{rgb}{0.7,0.3,0}

\definecolor{blue}{rgb}{.2,.6,.75}

\definecolor{green}{rgb}{.4,.7,.4}

\begin{document}

\title{Sharp bound for the fourth moment of holomorphic Hecke cusp forms}

\author{Peter Zenz}
\address{Department of Mathematics and Statistics, McGill University, 805 Sherbrooke St. W., \newline Montreal, QC H3A 2K6, Canada}
\email{peter.zenz@mail.mcgill.ca}


\date{\today}

\begin{abstract}
We prove that the fourth moment of holomorphic Hecke cusp forms is bounded provided that the Riemann Hypothesis holds for an appropriate degree 8 $L$-function. We accomplish this using Watson's formula, which translates the question in hand into a moment problem for $L$-functions which is amenable to the techniques of Soundararajan and Harper on obtaining sharp bounds for moments of the Riemann zeta function.
\end{abstract}

\maketitle

 \section{Introduction}
The distribution of mass of Hecke Maass cusp forms or of their holomorphic analogues is a central problem in analytic number theory. The  \textit{ Random Wave Conjecture} (RWC), introduced by Berry \cite{berryRegularIrregularSemiclassical1977}, suggests that as the eigenvalue tends to infinity, the Hecke Maass cusp form resembles a ``random wave". To be more concrete, as in \cite[Eq. 6.1]{hejhalTopographyMaassWaveforms1992}, we think of a random wave as the function given by 
$$\Psi(x+iy)= \sum_{n=1}^\infty c_n \sqrt{y} K_{iR}(2\pi n y) \cos(2\pi nx),$$ 
where the coefficients $c_n$ are chosen at random, with uniform distribution in $[-1,1]$. For comparison, the Fourier expansion of a Hecke Maass cusp form $f$ with spectral parameter $R$ is given by 
$$f(x + iy)= \sum_{n=1}^\infty \lambda_f(n) \sqrt{y} K_{iR}(2\pi n y) \cos(2\pi n x),$$
where $K_{iR}$ denotes the modified K-Bessel function and $\lambda_f(n)$ are the Hecke eigenvalues of $f$.

RWC predicts in particular that the moments of a Hecke Maass cusp form agree with the moments of a Gaussian random variable (see \cite[Conjecture 1.1]{humphriesEquidistributionShrinkingSets2018}).
Heijhal and Rackner gave a heuristic and numerical evidence toward this conjecture in \cite{hejhalTopographyMaassWaveforms1992}. Analogously we can formulate a conjecture for holomorphic Hecke cusp forms $f(z)$. In this case the Fourier expansion of $F(x+iy) =f(x+iy)y^{k/2}$ is given by
$$F(x+iy)=a_f(1) \sum_{n=1}^\infty \lambda_f(n)(4\pi n)^{(k-1)/2} e^{2\pi i n z}$$
where the constant $$|a_f(1)|^2=\frac{3}{\pi} \cdot \frac{ 4\pi \zeta(2)}{\Gamma(k) L(1, \sym^2f)}$$
arises to have $\brac{F,F}=1$. One can perform a similar heuristic as in the Maass form case using Theorem 3.5.2 \cite{salemPropertiesTrigonometricSeries1989} of Salem and Zygmund. They essentially show a central limit theorem for the partial sums (suitably normalized) of power series of the form
$$f(x)=\sum_{n=1}^\infty c_n \cdot a(n) e^{2\pi i n x},$$
where the $c_n$ are again chosen randomly, say with uniform distribution in $[-1,1]$ and some suitable coefficients $a(n)$. This suggests that $F(x+iy)$ is modelled by a complex Gaussian with mean 0 and variance $3/\pi$ (the inverse of the volume of the fundamental domain), as the weight $k$ tends to infinity.

For an analytic number theorist the fourth moment is of special interest because of its relation to $L$-functions. Indeed, as shown in  \cite[Eq. (2.7)]{blomerDistributionMassHolomorphic2013} and in restated in Lemma \ref{Watson}, an application of Watson's famous formula \cite[Theorem 3]{watsonRankinTripleProducts2008a} reduces the problem of finding an asymptotic formula for the fourth moment to understanding the first moment of a degree eight $L$-function. 

In \cite{buttcaneFourthMomentHecke2017}, Buttcane and Khan computed the asymptotic of the fourth moment of Hecke Maass cusp forms, assuming the \textit{Generalized Lindel\"of Hypothesis} (GLH), and confirmed a result predicted by RWC.  The analogous result for holomorphic cusp forms, which is the case we consider here, is still open and is in fact significantly harder. The reason for that is the size of the corresponding family of $L$-functions relative to the size of their conductor. For Maass forms we are averaging $T^2$ $L$-functions with analytic conductor of size $T^8$, hence the logarithmic ratio of those quantities is 4. On the other hand in the holomorphic cusp form case we are averaging $k$ $L$-functions, with conductor of size $k^6$, and so the logarithmic ratio is 6. One can compare this with evaluating the sixth moment of the Riemann zeta function, where the ratio is also 6 and no asymptotic is known under any reasonable conjecture like the \textit{Riemann Hypothesis} (RH). With an additional averaging over $f$ and $k$ and thus enlarging the family, Khan proved in \cite{khanFourthMomentHolomorphic2014} the desired asymptotic for the fourth moment of holomorphic cusp forms. Without averaging the best unconditional result is due to Blomer, Khan and Young. They showed  in \cite{blomerDistributionMassHolomorphic2013} that the fourth moment is bounded by $k^{1/3+\epsilon}$, where $k$ is the weight of the cusp form. Under GLH one can obtain trivially the bound $k^\epsilon$,  by bounding the $L$-function after applying Watson's formula. We improve on this conditionally on RH and show that the fourth moment is bounded.

To state our result, we write $S_k$ for the space of holomorphic Hecke cusp forms of weight $k$ on the full modular group $\Gamma=\SL_2(\mathbb{Z})$. Also, we write $B_{k}$ for a Hecke basis of $S_k$ and $\mathbb{H}$ for the usual upper half plane.  We abbreviate $F(x+iy)=f(x+iy)y^{k/2}$ and normalize $F$ so that 
$$\brac{F, F} =\int_{\Gamma \backslash \mathbb{H}} |f(z)|^2y^k \frac{dxdy}{y^2} = 1.$$
We establish the following result:
\begin{thm}\label{main}
Let $f$ be a holomorphic Hecke cusp form of even weight $k$, normalized so that $\brac{F,F}=1$. Assuming the Riemann Hypothesis for $L(s, f \times f \times g)$ and $L(s, \sym^2f)$, there exists a universal constant $C$ such that 
$$\int_{\Gamma \backslash \mathbb{H}} |f(z)|^4y^{2k} \frac{dxdy}{y^2}\leq C,$$
for $k$ large enough.
\end{thm}
\section{Notation}
Throughout,  $f(x)=O(g(x))$ or equivalently $f(x)\ll g(x)$ (or $g(x)\gg f(x)$) means that there exists an absolute constant $C>0$, such that $|f(x)|\leq C|g(x)|$ for all $x$ sufficiently large. The asymptotic equivalence $f(x)\sim g(x)$ means that $g(x)\neq0$ for sufficiently large $x$ and $\lim_{x\to \infty} f(x)/g(x)=1$. We say that two functions $f$ and $g$ have the same order of magnitude and write $f(x)\asymp g(x)$, if $f(x)\ll g(x)$ and $g(x)\gg f(x)$. For heuristics we might use the notation $f(x)\approx g(x)$, which should be interpreted only informally and indicates that $f(x)$ and $g(x)$ are roughly the same (up to some technical factors). The indicator function $1_P$ will equal $1$ if the statement $P$ is true and $0$ if it is false.

\section{$L$-functions}
We now gather same basic properties of the $L$-functions that are involved in this paper. Let $f$ be a Hecke cusp form of weight $k$ and let $\lambda_f(n)$ denote the $n$-th Hecke eigenvalue of $f$. The $L$-function associated to the cusp form $f$ is given by
$$L(s,f)= \sum_{n=1}^\infty \frac{\lambda_f(n)}{n^s}=\prod_p \Big(1- \frac{\alpha_f(p)}{p^s}\Big)^{-1}\Big(1-\frac{\beta_f(p)}{p^s}\Big)^{-1},$$
where $\alpha_f(p), \beta_f(p)=\overline{\alpha_f(p)}$ are complex numbers satisfying $|\alpha_f(p)|=|\beta_f(p)|=1$. Since $\lambda_f(p)=\alpha_f(p)+\beta_f(p)$, we have the Deligne bound $|\lambda_f(p)|\leq 2$. It is also worth mentioning that the Hecke eigenvalues are real. For $\Re{(s)}>1$ we define the symmetric square $L$-function by
$$L(s, \sym^2f) = \prod_{p} \Big(1-\frac{\alpha_f(p)}{p^s}\Big)^{-1}\Big(1-\frac{1}{p^s}\Big)^{-1}\Big(1-\frac{\beta_f(p)}{p^s}\Big)^{-1}$$ 
and the corresponding Gamma factor is given by $$L_\infty(s, \sym^2 f)=\Gamma_{\mathbb{R}}(s+1)\Gamma_{\mathbb{R}}(s+k-1)\Gamma_{\mathbb{R}}(s+k),$$
where $\Gamma_{\mathbb{R}}(s)=\pi^{-s/2}\Gamma(s/2).$ 
Shimura \cite{shimuraHolomorphyCertainDirichlet1975} showed that $L(s, \sym^2f)$ is entire and can be analytically continued to the entire complex plane. Moreover, the completed $L$-function satisfies the functional equation
$$L_\infty(s, \sym^2f)L(s, \sym^2f)=L_\infty(1-s)L(1-s, \sym^2f).$$

Let $g$ be another Hecke cusp form of weight $2k$ with Hecke eigenvalue (at primes) $\lambda_g(p)=\alpha_g(p)+\beta_g(p)$. Our main object of interest in this paper is the (degree 8) triple product $L$-function given by
\begin{align*}
L(s, f \times f \times g)=\prod_{p}&\Big(1-\frac{\alpha_f(p)\alpha_g(p)}{p^s}\Big)^{-1}\Big(1-\frac{\alpha_g(p)}{p^s}\Big)^{-2}\Big(1-\frac{\beta_f(p)\alpha_g(p)}{p^s}\Big)^{-1}\cdot\\
&\cdot \Big(1-\frac{\alpha_f(p)^2\beta_g(p)}{p^s}\Big)^{-1}\Big(1-\frac{\beta_g(p)}{p^s}\Big)^{-2}\Big(1-\frac{\beta_f(p)^2\beta_g(p)}{p^s}\Big)^{-1}
\end{align*} 
The corresponding Gamma factor is given by
\begin{align}\label{Gammafactor}
L_{\infty}(s, f \times f \times g)=&\Gamma_{\mathbb{R}}(s+2k-3/2)\Gamma_{\mathbb{R}}(s+2k-1/2)\Gamma_{\mathbb{R}}(s+k-1/2)^2 \cdot 
\\ &\cdot \Gamma_{\mathbb{R}}(s+k+1/2)^2\Gamma_{\mathbb{R}}(s+1/2)\Gamma_{\mathbb{R}}(s+3/2). \nonumber
\end{align}
The completed $L$-function is again entire, extends analytically to the entire complex plane and satisfies the functional equation
$$L_\infty(s, f\times f\times g)L(s, f\times f\times g)=L_\infty(1-s, f\times f\times g)L(1-s, f\times f\times g),$$ as can be seen from the work of Garrett \cite{garrettDecompositionEisensteinSeries1987}.

Let $B_{2k}$ be a basis of weight $2k$ Hecke cusp forms. For any $\mathcal{S}\subset B_{2k}$, it will be convenient to introduce the normalized sum
$$\sumh_{g \in \mathcal{S}}\lambda_g(n):=\frac{2\pi^2}{2k-1} \sum_{g\in \mathcal{S}} \frac{\lambda_g(n)}{L(1,\sym^2g)}.$$
Additionally, we define the normalized measure of $\mathcal{S}\subset B_{2k}$ by
$$\meas\{S\}:=\sumh_{g\in \mathcal{S}} 1 = \frac{2\pi^2}{2k-1} \sum_{g\in \mathcal{S}} \frac{1}{L(1,\sym^2g)}.$$

\section{Relation to $L$-functions and Heuristics}
As indicated before, the fourth moment of cusp forms is intimately related to $L$-functions. After we decompose $F^2$ into a Hecke basis of weight $2k$ cusp forms and apply Watson's formula (see \cite[Theorem 3]{watsonRankinTripleProducts2008a}), the problem translates into a question of bounding $L$-functions that are averaged over a family of $L$-functions. More precisely, we need to bound
\begin{equation}\label{interest}
 \frac{\pi^3}{2(2k-1)} \sum_{g \in B_{2k}} \frac{L(1/2, f\times f \times g)}{ L(1, \sym^2f)^2L(1, \sym^2g)}, 
\end{equation}
where $B_{2k}$ is a Hecke basis of weight $2k$ cusp forms.

The sum \eqref{interest} fits naturally into the setting of computing moments of $L$-functions. It is the first moment of a degree eight $L$-function at the central value $1/2$. Soundararajan obtained almost sharp bounds for the moments of the Riemann zeta function on the critical line in \cite{soundararajanMomentsRiemannZeta2009a}. Harper built upon these techniques, improved them in \cite{harperSharpConditionalBounds2013}, and achieved sharp bounds for the moments of zeta. Their methods are very robust and can also be applied to other $L$-functions. In particular, we can use them in our setting. 

Their approach is based on a probabilistic viewpoint of the logarithm of the zeta function. \mbox{Selberg's} Central Limit Theorem shows that the distribution  $\log |\zeta(\frac{1}{2}+it)|$  is approximately Gaussian with mean value 0 and variance $\frac{1}{2}\log \log T$ for $t\in [T, 2T]$ as $T\to \infty$. Similarly, we expect $\log L(1/2, f\times f \times g)$ to have roughly a Gaussian distribution of mean value $$-\sum_{p \leq k} \frac{\lambda_f(p)^4 + 4\lambda_f(p)^2-4}{2p}$$ and variance of size $$\sum_{p \leq k} \frac{\lambda_f(p)^4}{p}.$$ This can be seen by approximating the $L$-function by an Euler product and averaging over $g$ (see Lemma \ref{Sound} for more details).  For our problem, the dependency on $f$ is important, as $f$ depends on the weight $k$.   

With this in mind we can interpret $\sum_g^h \exp (\log L(1/2, f \times f \times g))$ as the expectation of the exponential of a random variable. For a Gaussian random variable X with mean $\mu$ and variance $\sigma^2$ we have the following computation:
\begin{align}
\mathbb{E}[e^X]&=\frac{1}{\sqrt{2\pi \sigma^2}} \int_{-\infty}^\infty \exp\bigg(t- \frac{(t-\mu)^2}{2\sigma^2}\bigg) dt \label{expectationrandom}\\
&=e^{\mu +\sigma^2/2} \frac{1}{\sqrt{2\pi \sigma^2}} \int_{-\infty}^\infty \exp \bigg(- \frac{(t- \mu - \sigma^2)^2}{2\sigma^2}\bigg) dt \nonumber \\
&=e^{\mu + \sigma^2/2}. \nonumber
\end{align}
For our scenario this heuristic yields
\begin{align*}\sumh_{g\in B_{2k}} \exp (\log L(1/2, g \times f \times f)) &\approx \exp \bigg(-\sum_{p \leq k} \frac{\lambda_f(p)^4 - 4\lambda_f(p)^2+4}{2p} + \frac{1}{2}  \sum_{p \leq k} \frac{\lambda_f(p)^4}{p}\bigg) \\
& = \exp\bigg(2 \sum_{p \leq k} \frac{\lambda_f(p)^2-1}{p}\bigg)\\
&\asymp L(1,\sym^2f)^2
\end{align*}
Inserting this bound into \eqref{interest} shows that the fourth moment of holomorphic cusp forms ought to be bounded.

\begin{remark}We want to highlight that the methods of this paper also apply to higher moments of the family of $L$-functions $L(s, f \times f \times g)$, when averaged over $g$. However, we cannot use this to our advantage for higher moments of holomorphic cusp forms: Watson's formula allows us to go from one world to the other only in the fourth moment setting.
\end{remark}

\begin{remark} We bound the fourth moment of the holomorphic cusp forms on the full fundamental domain. In \cite[inequality (1.2)]{blomerDistributionMassHolomorphic2013} it is shown that high moments of cusp forms on the full fundamental domain diverge. The reason for this are large values of holomorphic cusp forms high in the cusp.  For the fourth moment the measure of these large values is too small to have a significant effect. In fact, we believe the sixth moment of holomorphic cusp forms is at most $k^\epsilon$. More precisely, we put forward the following conjecture:
\end{remark}
\begin{conjecture}\label{moments} Let $f$ be a holomorphic cusp form of even weight $k$, normalized so that $\brac{F,F}=1$. Let $r$ be an even number and $y_0 >0$ then
$$P_r(y_0):=\int_{y_0}^\infty \int_0^1 |f(x+iy)y^{k/2}|^r \frac{dxdy}{y^2} \ll k^{\frac{r}{4}-\frac{3}{2}+\epsilon} + \frac{1}{y_0}.$$
\end{conjecture}
In \cite{blomerDistributionMassHolomorphic2013} Blomer, Khan and Young showed the lower bound $P_r(1)\gg k^{\frac{r}{4}-\frac{3}{2}-\epsilon}$, which matches the upper bound in Conjecture \ref{moments}. We note that their lower bound for $P_r(y_0)$ is essentially optimal in the range $y_0\gg \sqrt{k}$, as a cusp form is well approximated by only one Fourier coefficient high in the cusp (see also \cite{ghoshRealZerosHolomorphic2011b}). Conjecture \ref{moments} can be motivated for example by extending the proof of Theorem 1.8 in \cite{blomerDistributionMassHolomorphic2013} to higher moments and assuming square root cancellation in the shifted convolutions that are denoted by $T_f(l)$ in that paper. 
\begin{remark}
The divergence of high moments of holomorphic cusp forms on the full fundamental domain does not contradict the RWC. To make sense of the RWC for high moments one would need to restrict the integration range to a compact set, so that large values at the cusp are excluded. 
\end{remark}

\section{Proof of Theorem 1.1}
Recall the notation $F(z)= y^{k/2} f(z)$, where $f$ is an $L^2$-normalized Hecke cusp form of even weight $k$ on the full modular group $\SL_2(\mathbb{Z})$. Let $B_{2k}$ be a Hecke basis for the space of holomorphic cusp forms of weight $2k$. Write $g\in B_{2k}$ and $G(z)=y^k g(z)$. We have the following relation between moments of Hecke cusp forms and moments of L-functions:
\begin{lem}\label{Watson}
\begin{equation} \label{moment}
\int_{\Gamma \backslash \mathbb{H}} |f(z)|^4y^{2k} \frac{dxdy}{y^2}=\frac{\pi^3}{2(2k-1)} \sum_{g \in B_{2k}} \frac{L(1/2, f\times f \times g)}{ L(1, \sym^2f)^2L(1, \sym^2g)}.
\end{equation}
\end{lem}
\begin{proof}
Since $f^2$ is a cusp form of weight $2k$, we have the following decomposition in terms of Hecke eigenforms $g \in B_{2k}$:
$$\brac{F^2,F^2} = \sum_{g \in B_{2k}} |\brac{F^2, G}|^2.$$
At this point we apply Watson's formula (see \cite[Theorem 3]{watsonRankinTripleProducts2008a}) to the resulting inner product of three Hecke cusp forms (see also \cite[Eq. 2.7]{blomerDistributionMassHolomorphic2013}) so that
 $$ \sum_{g \in B_{2k}} |\brac{F^2, G}|^2=\frac{\pi^3}{2(2k-1)} \sum_{g \in B_{2k}} \frac{L(1/2, f\times \bar{f} \times g)}{ L(1, \sym^2f)^2L(1, \sym^2g)}.$$ Finally, we drop the complex conjugation bar of $f$, as the Fourier coefficients of $f$ are real, and the lemma follows.
\end{proof}
Now that we have reduced Theorem \ref{main} to bounding an average of $L$-functions we will follow the approach of Soundararajan and Harper to control the right-hand side of \eqref{moment}. At first we need to approximate our $L$-function $L(1/2, f \times f \times g)$ with a short Dirichlet polynomial over primes. Working with this Dirichlet polynomial will enable us to detect the underlying Gaussian behaviour of $\log L(1/2, f\times f \times g)$. To accomplish this, we use an idea of Soundararjan \cite{soundararajanMomentsRiemannZeta2009a} as adapted by Chandee \cite{chandeeExplicitUpperBounds2009} to our context.

\begin{lem}\label{Sound}Let $f$ and $g$ be Hecke cusp forms of even weight $k$ and $2k$, respectively, for the full modular group. Assuming the Riemann Hypothesis for $L(1/2, f\times f \times g)$, we have for any $x\geq 2$
\begin{align*}
\log L(1/2, f \times f \times g) \leq &\sum_{p \leq x} \frac{\lambda_f(p)^2\lambda_g(p)}{p^{1/2+1/\log x}} \frac{\log (x/p)}{\log x} \\&
+ \sum_{p^2\leq x} \frac{(\lambda_f(p)^4-4\lambda_f(p)^2+4)(\lambda_g(p^2)-1)}{2p^{1+2/\log x}}\frac{\log (x/p^2)}{\log x} + \frac{\log k^6}{\log x} + O(1)
\end{align*}\end{lem}
\begin{proof}
We express the Hecke eigenvalues $\lambda(p)$ of $f$ and $g$ in terms of their Satake parameters $\alpha(p)$ and $\beta(p)$, more precisely $\lambda_f(p)=\alpha_f(p)+\beta_f(p)$ and $\lambda_g(p)=\alpha_g(p)+\beta_g(p)$. Now we can directly apply Theorem 2.1 in \cite{chandeeExplicitUpperBounds2009} with $c=1$ and get

\begin{align*}\log L(1/2, f\times f \times g) \leq &\sum_{\ell=1}^\infty\sum_{p^\ell \leq x} \frac{(\alpha_f(p)^{2\ell}+\beta_f(p)^{2\ell}+2)(\alpha_g(p)^\ell+\beta_g(p)^\ell)}{\ell p^{(\frac{1}{2}+\frac{1}{\log x})\ell}}\frac{\log (x/p^\ell)}{\log x} \\
&+\frac{\log k^6}{\log x} +O\left(\frac{1}{\log^2x}\right)
\end{align*}
Here we used that the analytic conductor of $L(1/2, f\times f \times g)$ is of size $k^6$, which can be seen from Gamma factor $L_\infty(s, f\times f \times g)$ in \eqref{Gammafactor}. 
By the Deligne bound $|\lambda_f(p)| \leq 2$ the contribution of the prime powers $p^\ell$ with $\ell\geq 3$ can be shown to be $O(1)$. Since $\alpha_f(p)^2+\beta_f(p)^2+2=\lambda_f(p)^2$, $\alpha_f(p)^4+\beta_f(p)^4+2=\lambda_f(p)^4-4\lambda_f(p)^2+4$ and $\alpha_g(p)^2+\beta_g(p)^2=\lambda_g(p)^2-2=\lambda_g(p^2)-1$ the lemma follows.
\end{proof}

 \begin{remark}Notice that on average over $g$ the coefficients $\lambda_g(p)$ and $\lambda_g(p^2)$ are close to $0$. Consequently, we expect the mean value of $\log L(1/2, f\times f \times g)$ to be essentially 
$$\sum_{p^2\leq x} \frac{-(\lambda_f(p)^4-4\lambda_f(p)^2+4)}{2p^{1+2/\log x}}.$$
In contrast to that  $\lambda_g(p)^2$ is close to 1 on average and so the variance should be
$$\sumh_{g\in B_{2k}}\left(\sum_{p \leq x} \frac{\lambda_f(p)^2\lambda_g(p)}{p^{1/2+\log x}}\right)^2 \sim \sum_{p\le x} \frac{\lambda_f(p)^4}{p^{1+2/\log x}}.$$
\end{remark}
\subsection{Detecting Randomness}
In his proof, Harper detected the randomness of the harmonics $\Re(p^{-it})$ with  Proposition 2 in \cite{harperSharpConditionalBounds2013}. For our harmonics $\lambda_g(p)$, the role will be played by the following version of Petersson's Trace Formula:

\begin{lem}[Petersson Trace Formula] \label{Petersson}
Let $k$ be large and let  $n=p_1^{\alpha_1} \cdots p_r^{\alpha_r} \leq k^2/10^4$, where the $p_i$ are distinct primes and $\alpha_i \in \mathbb{N}$ for all $i$. Then
\begin{equation}\label{Peterssoneasy}\sumh_{g \in B_{2k}} \prod_{i=1}^r \lambda_g(p_i)^{\alpha_i}= h_1(n) + O(k^3e^{-k})
\end{equation}
where 
$$h_1(n):= \prod_{i=1}^r \frac{1_{2|\alpha_i} \cdot (\alpha_i)!}{((\alpha_i/2)!)^2(\alpha_i/2+1)},$$
in particular, $h_1(n)=0$ if any of the exponents $\alpha_i$ is odd.

Moreover, if $n=p_1^{\beta_1} \cdots p_r^{\beta_r} \leq k/100$, with $p_i$ distinct primes and $\beta_i \in \mathbb{N}$ for all $i$, then
$$\sumh_{g \in B_{2k}} \prod_{i=1}^r \lambda_g(p_i^2)^{\beta_i}= h_2(n) + O(k^4e^{-k}),$$
where $$h_2(n)=\prod_{i=1}^r \sum_{k=0}^{\beta_i} \binom{\beta_i}{k}(-1)^k \frac{(2(\beta_i-k))!}{(\beta_i-k)!(\beta_i-k+1)!}$$
In particular, $h_2(n)=0$ if  $\beta_i=1$ for some $i$, and $h_2(n)\leq \prod_{i=1}^r 3^{\beta_i}$ in general. 

We also have the following combined result: Let $a=p_1^{\alpha_1} \cdots p_r^{\alpha_r}$, $b=q_1^{\beta_1} \cdots q_s^{\beta_s}$, with $a \cdot b^2 \leq k^2/10^4$, $p_i$ and $q_j$ all distinct from each other. Then
\begin{equation}\label{combilemma}\sumh_{g\in B_{2k}} \prod_{i=1}^r \lambda_g(p_i)^{\alpha_i} \prod_{j=1}^s \lambda_g(q_j^2)^{\beta_i}=h_1(a)h_2(b)+O(k^5e^{-k}).
\end{equation}
\end{lem}
\begin{remark}Notice that $h_1$ is a multiplicative function supported on even numbers. This is reminiscent of the correlations of powers of independent Gaussian random variables. The multiplicativity of $h_1$ should be interpreted as quasi-independence and the support on even numbers reminds us that odd moments of Gaussian random variables vanish. We also highlight the condition that $n \leq k^2/10^4$. The total number of available harmonics is $k$, hence the length of the square of the Dirichlet polynomial should not exceed $k^2$, so that the only contribution comes from the main term. The bound $h_2(n)\leq \prod_{i=1}^r 3^{\beta_i}$ follows upon noting that $\lambda_g(p^2)\leq 3$ by the Deligne bound.
\end{remark}
\begin{proof}
We want to use the Petersson Trace Formula in the form of Lemma 2.1 in \cite{rudnickLowerBoundsMoments2006} which says that
\begin{equation}\label{PeterssonTrace}\sumh_{g \in B_{2k}} \lambda_g(t)\lambda_g(u)= 1_{t=u} + O(e^{-k}),
\end{equation}
 if $k$ is large and $t$ and $u$ are natural numbers with $tu \leq k^2/10^4$.
\newline
\newline
 To do so we need to express $\lambda_g(p_i)^{\alpha_i}$ in terms of $\lambda_g(p^\ell)$ for $0 \leq \ell \leq \alpha_i$. This can be achieved via the Hecke relations of the Fourier coefficients. An easy computation, as done in Lemma 7.1 of \cite{lauSUMSFOURIERCOEFFICIENTS2011}, shows that 
\begin{equation}\label{lambdap}\lambda_f(p)^{\alpha}= \Big(A_{\alpha} + \sum_{\ell=1}^{\alpha/2} C_{\alpha}(\ell) \lambda_f(p^{2\ell})\Big)1_{2|\alpha}+ \Big( B_{\alpha}\lambda_f(p) + \sum_{\ell=1}^{\alpha/2-1} D_{\alpha}(\ell) \lambda_f(p^{2\ell+1})\Big)1_{2|\alpha+1}
\end{equation}
with $$A_{\alpha}= \frac{(\alpha)!}{((\alpha/2)!)^2(\alpha/2+1)},\quad \quad C_{\alpha}(\ell)= \frac{(\alpha)!(2\ell+1)}{(\alpha/2-r)!(\alpha/2+r+1)!},$$
(these coefficients only appear in the expression of $\lambda_f(p)^{\alpha}$ when $\alpha$ is even)
$$B_{\alpha}= \frac{2(\alpha)!}{((\alpha-1)/2)!((\alpha+3)/2)!} \quad \text{and} \quad D_{\alpha}(\ell)=\frac{(\alpha)!(2\ell+2)}{((\alpha-1)/2)-\ell)!((\alpha+3)/2+\ell)!}$$
(these coeffiecients only appear in the expression of $\lambda_f(p)^{\alpha}$ when $\alpha$ is odd).
It follows that the left-hand side of equation \eqref{Peterssoneasy} is given by
$$\sumh_{g\in B_{2k}} \prod_{i=1}^r \bigg\{\Big(A_{\alpha_i}+\sum_{\ell_1=1}^{\alpha_i/2}C_{\alpha_i}(\ell_1)\lambda_f(p_i^{2\ell_1})\Big)1_{2|\alpha_i}+\Big(B_{\alpha_i}\lambda_f(p_i)+\sum_{\ell_2=1}^{\alpha_i/2-1}D_{\alpha_i}(\ell_2)\lambda_f(p_i^{2\ell+1})\Big)1_{2|(\alpha_i+1)}\bigg\}$$
We apply identity \eqref{PeterssonTrace} and get the main term 
$$\prod_{i=1}^r A_{\alpha_i} 1_{2|\alpha_i}=\prod_{i=1}^r \frac{(\alpha_i)!}{((\alpha_i/2)!)^2(\alpha_i/2+1)}1_{2|\alpha_i},$$
since the primes $p_i$ are distinct for different $1\leq i \leq r$. To bound the error term we first notice that
$A_{\alpha}\leq 2^{\alpha}$, $\sum_{\ell=1}^{\alpha/2}C_{\alpha}(\ell)\leq 2^{\alpha}, B_{\alpha} \leq 2 \cdot 2^\alpha$ and $\sum_{\ell=1}^{(\alpha-1)/2} D_{\alpha} \leq 2\cdot 2^{\alpha}$. 
Consequently, the error term is bounded by
$$O\Big(e^{-k} \prod_{i=1}^r 4 \cdot 2^{\alpha_i}\Big)=O\big(e^{-k}k^3\big).$$
Here we also used the crude bounds $4^{r}\ll k$ and $\prod_{i=1}^r 2^{\alpha_i} \leq k^2$. This shows the first part of the lemma. 

The second part of the lemma follows similarly upon using $\lambda_g(p_i^2)=\lambda_g(p_i)^2-1$ and the binomial theorem. More precisely, we have
\begin{align*}\sumh_{g \in B_{2k}} \prod_{i=1}^r \lambda_g(p_i^2)^{\beta_i}&=\sumh_{g\in B_{2k}} \prod_{i=1}^r (\lambda_g(p_i)^2-1)^{\beta_i} \\
&=\sumh_{g\in B_{2k}} \prod_{i=1}^r \sum_{\ell=0}^{\beta_i} \binom{\beta_i}{\ell} (-1)^{\ell} \lambda_g(p_i)^{2(\beta_i-\ell)}\\
&=\prod_{i=1}^r\sum_{\ell_i=0}^{\beta_i}  \Big(\binom{\beta_i}{\ell_i}(-1)^{\ell_i}\Big) \sumh_{g\in B_{2k}}\prod_{i=1}^r \lambda_g(p_i)^{2(\beta_i-\ell_i)}.
\end{align*}
At this point we use relation \eqref{lambdap} to rewrite $\lambda_g(p_i)^{2(\beta_i-\ell_i)}$ in terms of $\lambda_g(p_i^\ell)$ for $0\leq \ell \leq 2(\beta_i-\ell_i)$. Again, an application of the Petersson Trace formula (see formula \eqref{PeterssonTrace}), yields the main term
$$\prod_{i=1}^r \sum_{\ell=0}^{\beta_i}(-1)^{\ell} \frac{(2(\beta_i-\ell))!}{(\beta_i-\ell)!(\beta_i-\ell+1)!}$$
as desired. The error term is given by
$$O\Big(e^{-k} \cdot \prod_{i=1}^r \sum_{\ell=0}^{\beta_i} \binom{\beta_i}{\ell} \cdot \Big(A_{2(\beta_i-\ell)}+ \sum_{m=1}^{\beta_i-\ell}C_{2(\beta_i-\ell)}(m)\Big)\Big)=O\Big(e^{-k} 2^r \prod_{i=1}^r 5^{\beta_i}\Big),$$
as $A_{2(\beta_i-\ell)}\leq 4^{\beta_i-\ell}$, $\sum_{m=1}^{\beta_i-\ell}C_{2(\beta_i-\ell)}(m)\leq 4^{\beta_i-\ell}$ and
$$\sum_{\ell=0}^{\beta_i}\binom{\beta_i}{\ell}4^{\beta_i-\ell}=5^{\beta_i}.$$
So the contribution of the error term is given by $$O\bigg( e^{-k} 2^r \prod_{i=1}^r 5^{\beta_i}\bigg)=O(e^{-k}k^4).$$

Finally, the last part of the lemma, namely equation \eqref{combilemma}, follows in a similar vein. The main term is, as desired, given by
$$\prod_{i=1}^r \frac{1_{2|\alpha_i} \cdot (\alpha_i)!}{((\alpha_i/2)!)^2(\alpha_i/2+1)} \cdot \prod_{j=1}^s  \sum_{k=0}^{\beta_j} \binom{\beta_j}{k}(-1)^k \frac{(2(\beta_j-k))!}{(\beta_j-k)!(\beta_j-k+1)!}.$$
The error term is now given by
$$O\bigg(e^{-k} 2^{r+s} \prod_{i=1}^r 2^{\alpha_i} \cdot \prod_{j=1}^s 5^{\beta_i}\bigg)=O(e^{-k}k^5).$$
This concludes the proof of the entire lemma.
\end{proof}

\begin{lem}\label{combinato} Define the function $h_1$ as in Lemma \ref{Petersson} and let $u(p)$ be any real numbers. For any numbers $x_1, x_2\geq 1$, we have
\begin{equation}\Big|\sum_{x_1 < p_1, \ldots, p_n \leq x_2} \frac{u(p_1) \cdots u(p_n)}{\sqrt{p_1 \cdots p_n}} h_1(p_1 \cdots p_n )\Big| \leq \frac{n!}{2^{n/2} (n/2)!} \bigg( \sum_{x_1 < p \leq x_2} \frac{u(p)^2}{p}\bigg)^{\frac{n}{2}}\label{combi}\end{equation}
if $n$ is even and $0$ if $n$ is odd.
\end{lem}
\begin{proof}
Let $U$ denote the sum on the left-hand side of the desired inequality \eqref{combi}. Recall that $h_1$ is supported only on squares. In particular, $n$ has to be even and we write $n=2\ell$, so that  
\begin{equation*}\label{notdistinct}
U=\sum_{x_1< p_1, \ldots, p_{2\ell}\leq x_2} \frac{u(p_1)\cdots u(p_{2\ell})}{\sqrt{p_1 \cdots p_{2\ell}}} h_1(p_1 \cdots p_{2\ell}) .
\end{equation*}
We now write $p_1 \cdots p_{2\ell}=q_1^{\alpha_1} \cdots q_r^{\alpha_r}$, where the primes $q_i$ for $1\leq i\leq r$ are distinct and $\alpha_i\geq 1$ for all $1\leq i \leq r$. Then $U$ equals
\begin{equation*}\label{nowdistinct}
\sum_{1\leq r \leq 2\ell} \sum_{\alpha_1 + \cdots +\alpha_r=2\ell} \sum_{x_1 < q_1 < \ldots < q_r \leq x_2 }\binom{2\ell}{\alpha_1, \ldots, \alpha_r} \frac{u(q_1)^{\alpha_1} \cdots u(q_r)^{\alpha_r}}{\sqrt{q_1^{\alpha_1}\cdots q_r^{\alpha_r}}} h_1(q_1^{\alpha_1}\cdots q_r^{\alpha_r}),
\end{equation*}
where the multinomial coefficient counts the number of representations such that $p_1\cdots p_{2\ell}=\prod_{1\leq i \leq r} q_i^{\alpha_i}$.
Since $h_1$ is supported only on squares we see that $\alpha_i$ for $1\leq i \leq r$ is divisible by 2 and consequently $r\leq \ell$. It follows that
\begin{align*}
U=&\sum_{1\leq r\leq \ell} \sum_{\substack{\alpha_1 + \cdots + \alpha_r = 2\ell \\ 2 |\alpha_i}}\sum_{x_1< q_1 < \ldots < q_r \leq x_2}\binom{2\ell}{\alpha_1, \cdots, \alpha_r} \frac{u(q_1)^{\alpha_1} \cdots u(q_r)^{\alpha_r}}{\sqrt{q_1^{\alpha_1}\cdots q_r^{\alpha_r}}} h_1(q_1^{\alpha_1}\cdots q_r^{\alpha_r}) \\
=&\sum_{1\leq r\leq \ell} \sum_{\substack{\beta_1+\cdots +\beta_r=\ell \\ \beta_i \geq 1}} \sum_{x_1< q_1 < \ldots < q_r\leq x_2} \binom{2\ell}{2\beta_1, \ldots, 2\beta_r}\frac{u(q_1)^{2\beta_1} \cdots u(q_r)^{2\beta_r}}{q_1^{\beta_1}\cdots q_r^{\beta_r}} \prod_{1\leq i \leq r} \frac{(2\beta_i)!}{\beta_i! (\beta_i+1)!}
\end{align*}
We simplify and use the bound $(\beta_i+1)!\geq 2^{\beta_i}$ so that
\begin{align*}
U\leq &\frac{(2\ell)!}{\ell !} \sum_{1\leq r\leq \ell} \sum_{\substack{\beta_1+\cdots +\beta_r=\ell \\ \beta_i \geq 1}} \sum_{x_1 < q_1 < \ldots < q_r \leq x_2} \binom{\ell}{\beta_1, \cdots, \beta_r} \frac{u(q_1)^{2\beta_1} \cdots u(q_r)^{2\beta_r}}{q_1^{\beta_1}\cdots q_r^{\beta_r}} \frac{1}{2^{\beta_1} \cdots 2^{\beta_r}} \\
=&\frac{(2\ell)!}{\ell!} \bigg( \sum_{x_1< q \leq x_2} \frac{u(q)^2}{2q}\bigg)^{\ell}
\end{align*}
This concludes the proof of the lemma as
$$\Big|\sum_{x_1 < p_1, \ldots p_n \leq x_2} \frac{u(p_1) \cdots u(p_n)}{\sqrt{p_1 \cdots p_n}} h_1(p_1 \cdots p_n )\Big| \leq \frac{n!}{2^{n/2} (n/2)!} \left( \sum_{x_1 < p \leq x_2} \frac{u(p)^2}{p}\right)^{n/2}$$ for $n$ even. The function $h_1$ is not supported on odd powers and so the quantity $U$ is clearly zero if $n$ is odd.
\end{proof}
\begin{lem}\label{combinato2}
Let $w(p)$ be any real numbers such that $|w(p)|\leq C$ and define the function $h_2$ as in Lemma \ref{Petersson}. Then
\begin{equation} \Big| \sum_{2^m < p_1, \ldots, p_{2M} \leq 2^{m+1}} \frac{w(p_1) \cdots w(p_{2M})}{p_1 \cdots p_{2M}} \cdot h_2(p_1 \cdots p_{2M})\Big| \leq \frac{(2M)!}{M!} \left(\frac{72C^2}{2^m}\right)^M\label{combi2}\end{equation} 
\end{lem}
\begin{proof}
The main difference to Lemma \ref{combinato} is that the function $h_2$ is supported on integers that are divisible by squares, rather than integers that are squares. This leads to more difficult combinatorics. Let $W$ denote the sum on the left-hand side of inequality \eqref{combi2}. As in Lemma \ref{combinato} we express $p_1 \cdots p_{2M}$ in terms of distinct primes, i.e. $p_1\cdots p_{2M}=\prod_{i=1}^r q_i^{\alpha_i}$. Then
\begin{equation*}W=\sum_{1\leq r \leq M} \sum_{\substack{\alpha_1 + \cdots + \alpha_r=2M \\ \alpha_i \geq 2}} \sum_{2^m< q_1 < \ldots < q_r \leq 2^{m+1}}\binom{2M}{\alpha_1, \ldots, \alpha_r}\frac{w(q_1)^{\alpha_1} \cdots w(q_r)^{\alpha_r}}{q_1^{\alpha_1} \cdots q_r^{\alpha_r}}h_2(q_1^{\alpha_1} \cdots q_r^{\alpha_r}). \label{squarecombi}
\end{equation*}
Here we used that $h_2(q_1^{\alpha_1} \cdots q_r^{\alpha_r})$ is zero if $\alpha_i=1$ for some $i$. Next, we apply the crude bound $h_2(q_1^{\alpha_1} \cdots q_r^{\alpha_r})\leq \prod_{1\leq i \leq r} 3^{\alpha_i}$ and $|b(p_i)| \leq C$ for $1\leq i \leq r$. It follows that $|W|$ is bounded by
\begin{equation*}\label{order}
(3C)^{2M} \sum_{1\leq r \leq M}  \sum_{\substack{\alpha_1 + \cdots + \alpha_r=2M \\ \alpha_i \geq 2}} \binom{2M}{\alpha_1, \ldots, \alpha_r} \sum_{2^m < q_1 < \ldots < q_r \leq 2^{m+1}} \frac{1}{q_1^{\alpha_1}\cdots q_r^{\alpha_r}}
\end{equation*}
We can de-order the primes and drop the condition that they are distinct so that
\begin{align*}
|W|\leq &(3C)^{2M} \sum_{1\leq r \leq M}  \sum_{\substack{\alpha_1 + \cdots + \alpha_r=2M \\ \alpha_i \geq 2}} \binom{2M}{\alpha_1, \ldots, \alpha_r} \frac{1}{r!}\prod_{1\leq i \leq r} \bigg( \sum_{2^m < q_i \leq 2^{m+1}} \frac{1}{q_i^{\alpha_i}} \bigg) \\
\leq & (3C)^{2M} \sum_{1\leq r \leq M}  \sum_{\substack{\alpha_1 + \cdots + \alpha_r=2M \\ \alpha_i \geq 2}} \frac{2^{mr}}{2^{m\cdot 2M}} \frac{(2M)!}{\alpha_1! \cdots \alpha_r!} \frac{1}{r!} \\
\leq &(3C)^{2M} \frac{(2M)!}{2^{m\cdot 2M}}  \sum_{1\leq r \leq M}  \sum_{\substack{\alpha_1 + \cdots + \alpha_r=2M \\ \alpha_i \geq 2}}\frac{2^{mr}}{r!}
\end{align*}
Comparing the ratios of consecutive terms of the sequence $2^{rm}/r!$ we see that the sequence is increasing and so its maximum is attained when $r=M$. Together with the trivial bound
$$ \sum_{1\leq r \leq M}  \sum_{\substack{\alpha_1 + \cdots + \alpha_r=2M \\ \alpha_i \geq 2}} 1 \leq M \cdot 2^{2M}\leq 2^{3M}$$ we conclude that
$$|W|\leq (72C^2)^M \frac{(2M)!}{M!} \frac{1}{2^{mM}},$$
which completes the proof of the lemma.
\end{proof}
We know that the expectation of a product of independent random variables is equal to the product of the expectations. The following lemma reminds us of this fact in our specialized setting. 
\begin{lem}\label{Gaussian} Let $u(p), w(p)$ be any real numbers such that $|u(p)|\leq p^{1/2}$ and $|w(p)|\leq C \leq p$, for a constant $C\geq0$. Suppose $k$ is large, fix the real numbers $1\leq y_{i-1} < y_i$ for $1 \leq i \leq I$ and let $n_i, m, M$ be positive integers such that $ 2^{(m+1)\cdot 2M}\prod_{i=1}^I y_i^{n_i}  \leq k^2/10^4$. Moreover, let  $M\leq 2^m$ and $2^{m+1} \leq y_0$ if $M\neq0$, then
\begin{align*}&\sumh_{g\in B_{2k}} \prod_{1\leq i \leq I} \bigg( \sum_{y_{i-1} < p \leq y_i} \frac{u(p)\lambda_g(p)}{p^{1/2}}\bigg)^{n_i} \cdot \bigg( \sum_{2^m < q \leq 2^{m+1}} \frac{w(q)\lambda_g(q^2)}{q}\bigg)^{2M} \\
\ll &\prod_{1\leq i \leq I} \frac{1_{2|n_i} \cdot n_i!}{2^{n_i/2} (n_i/2)!} \bigg( \sum_{y_{i-1} < p \leq y_i} \frac{u(p)^2}{p}\bigg)^{\frac{n_i}{2}} \cdot \frac{(2M)!}{M!} \bigg(\frac{72C^2}{2^m}\bigg)^M + k^7e^{-k}
\end{align*}
\end{lem}
\begin{proof}
We want to apply the Petersson Trace Formula to detect the random behaviour of the coefficients $\lambda_g(p)$ and $\lambda_g(p^2)$.  We start by expanding the $n_i$-th and $2M$-th powers.
\begin{align*}&\sumh_{g\in B_{2k}} \prod_{1\leq i \leq I} \bigg( \sum_{y_{i-1} < p \leq y_i} \frac{u(p)\lambda_g(p)}{p^{1/2}}\bigg)^{n_i} \cdot \bigg( \sum_{2^m < q \leq 2^{m+1}} \frac{w(q)\lambda_g(q^2)}{q}\bigg)^{2M} \\
=&\sumh_{g\in B_{2k}} \prod_{1\leq i \leq I} \bigg( \sum_{y_{i-1} < p_1, \ldots, p_{n_i} \leq y_i} \prod_{1\leq r \leq n_i} \frac{u(p_r)\lambda_g(p_r)}{p_r^{1/2}}\bigg)\cdot \bigg(\sum_{2^m < q_1, \ldots, q_{2M} \leq 2^{m+1}} \prod_{1\leq s \leq 2M} \frac{w(q_s)\lambda_g(q_s^2)}{q_s}\bigg)
\end{align*}
Next we expand the product over $i$ and interchange the order of summation. We get
\begin{equation}\sum_{\tilde{p}}  \sum_{\tilde{q}} C(\tilde{p}) D(\tilde{q})\cdot \sumh_{g\in B_{2k}} \prod_{1\leq i \leq I} \prod_{1\leq r \leq n_i} \lambda_g(p_{i,r})\prod_{1\leq s\leq 2M} \lambda_g(q_s^2),\label{beforePetersson}
\end{equation}
where
$$C(\tilde{p})=\prod_{1\leq i \leq I} \bigg( \prod_{1\leq r \leq n_i} \frac{u(p_{i,r})}{p_{i,r}^{1/2}}\bigg) \quad \text{and} \quad D(\tilde{q})= \prod_{1\leq s\leq 2M} \frac{w(q_s)}{q_s}$$
with $\tilde{p}=(p_{1,1}, p_{1,2}, \ldots p_{1,n_1}, p_{2,1}, \ldots p_{2, n_2}, \ldots p_{I, n_I})$ and $\tilde{q}=(q_1, \ldots q_{2M})$ . Each component of the vectors $\tilde{p}$ and $\tilde{q}$ is prime and they satisfy the conditions
$$y_{i-1} < p_{i,1}, \ldots, p_{i, I} \leq y_i \quad\forall{1\leq i \leq I}\quad \text{and} \quad 2^m < q_1, \ldots q_{2M} \leq 2^{m+1}.$$
By our assumption $\prod_{i=1}^I y_i^{n_i} \cdot 2^{(m+1)\cdot 2M} \leq k^2/10^4$ and since $2^{m+1} \leq y_0$ the primes $p_{i,r}$ are distinct from the primes $q_s$.  Hence we can apply the Petersson Trace Formula, namely Lemma \ref{Petersson}, and get
$$\sumh_{g\in B_{2k}} \prod_{1\leq i \leq I} \prod_{1\leq r \leq n_i} \lambda_g(p_{i,r})\prod_{1\leq s\leq 2M} \lambda_g(q_s^2)=h_1\bigg(\prod_{1\leq i \leq I} \prod_{1\leq r \leq n_i} p_{i,r}\bigg) \cdot h_2 \bigg( \prod_{1\leq q \leq 2s} q_s\bigg) +O(k^5e^{-k}).$$
It follows that expression \eqref{beforePetersson} is equal to
$$\sum_{\tilde{p}} \sum_{\tilde{q}} C(\tilde{p}) D(\tilde{q})\cdot h_1\Big(\prod_{1\leq i \leq I} \prod_{1\leq r \leq j_i} p_{i,r}\Big) h_2\Big(\prod_{1\leq s \leq 2M} q_s\Big) + O\Big(e^{-k}k^5 \sum_{\tilde{p}}\sum_{\tilde{q}}|C(\tilde{p})D(\tilde{q})|\Big).$$
To bound the main term we notice that there is no dependency on the cusp forms $g$ anymore and so we can analyze the sums over $\tilde{p}$ and $\tilde{q}$ separately. We begin with the summation over $\tilde{p}$ and use the multiplicativity of $h_1(n)$  so that this part of the main term equals in absolute value
\begin{align}
&\Big|\sum_{\tilde{p}} C(\tilde{p})\cdot h_1\bigg(\prod_{1\leq i \leq I } \prod_{1\leq r \leq j_i} p_{i,r}\bigg) \label{primesbefore}\Big|\\
=& \Big|\prod_{1\leq i \leq I} \bigg( \sum_{y_{i-1} < p_{i,1}, \ldots p_{i, n_i} \leq y_i} \frac{u(p_{i,1}) \cdots u(p_{i,n_i})}{\sqrt{p_{i,1}\cdots p_{i,n_i}}} \cdot h_1(p_{i,1} \cdots p_{i, n_i})\bigg)\Big|\nonumber\\
\leq &\prod_{1\leq i \leq I} \frac{1_{2|n_i} \cdot n_i!}{2^{n_i/2} (n_i/2)!} \bigg( \sum_{y_{i-1} < p \leq y_i} \frac{u(p)^2}{p}\bigg)^{\frac{n_i}{2}}, \label{primescontri}
\end{align}
by Lemma \ref{combinato}.
Similarly, for the sum over $\tilde{q}$ we get
\begin{align}
&\Big|\sum_{\tilde{q}} D(\tilde{q}) \cdot h_2\bigg( \prod_{1\leq s \leq 2M}q_s\bigg) \Big| \nonumber \\
=&\Big| \sum_{2^m < q_1, \ldots, q_{2M} \leq 2^{m+1}} \frac{w(q_1) \cdots w(q_{2M})}{q_1 \cdots q_{2M}} \cdot h_2(q_1\cdots q_{2M}) \Big|\nonumber\\
\leq & \frac{(2M)!}{M!} \left(\frac{72C^2}{2^m}\right)^M \label{primessquaredcontri},&
\end{align}
where we used Lemma \ref{combinato2}. It remains to control the error term, which is given by
\begin{align*}
&O\Big(e^{-k}k^5 \sum_{\tilde{p}} \sum_{\tilde{q}} |C(\tilde{p})D(\tilde{q})|\Big) \\
=& O\bigg(e^{-k}k^5 \prod_{1\leq i \leq I}\bigg( \sum_{y_{i-1} < p \leq y_{i}} \frac{|u(p)|}{p^{1/2}}\bigg)^{n_i} \cdot \bigg( \sum_{2^m < q \leq 2^{m+1}} \frac{|w(q)|}{q}\bigg)^{2M}\bigg)\\
=& O(e^{-k} k^7).
\end{align*}
In the last line we used $|u(p)|\leq p^{1/2}$, $|w(p)|\leq p$ and the condition $\prod_{i=1}^I y_{i}^{n_i} \cdot 2^{(m+1)\cdot 2M} \leq k^2/10^4$. Inserting \eqref{primescontri} and \eqref{primessquaredcontri} into \eqref{beforePetersson} together with the error term calculation concludes the proof of Lemma \ref{Gaussian}.
\end{proof}

\subsection{Setup}\label{Setup}
Recall that Lemma \ref{Sound} tells us essentially that
\begin{equation}\label{Soundeasy}
L(1/2, f\times f \times g) \ll \exp\Big(\sum_{p\leq x} \frac{\lambda_f(p)^2 \lambda_g(p)}{p^{1/2}}\Big)\exp\Big(\sum_{p^2\leq x} \frac{(\lambda_f(p)^4-4\lambda_f(p)^2+4)\cdot \lambda_g(p^2)}{2p}\Big).
\end{equation}

 We will perform a Taylor expansion on the exponentials and so it is important to control the size of the Dirichlet polynomials. This is quite technical, and here it is how we do it precisely:
\newline
\newline
Define the sequence  $(\beta_i)_{i\geq0}$ by
\begin{equation}\label{beta} \beta_0:= 0;~~\beta_i:= \frac{20^{i-1}}{(\log \log k)^2}~\text{ for all }i \geq 1,
\end{equation}
and $$I=I_k :=1 + \max\{i \colon \beta_i \leq e^{-10000}\}.$$
To simplify notation write
\begin{equation} \label{coefficients} x_j:=k^{\beta_j} \quad \text{and} \quad u_{f,j}(p):= \frac{\lambda_f(p)^2}{p^{1/(\beta_{j} \log k)}} \frac{\log (x_j/p)}{\log x_j}\leq \lambda_f(p)^2
\end{equation}
For each $1\leq i \leq j \leq I$ define
$$G_{(i,j)}(g):= \sum_{x_{i-1} < p \leq x_{i}} \frac{\lambda_f(p)^2\lambda_g(p)}{p^{1/2 + 1/(\beta_j \log k)}} \frac{\log (x_j/p)}{\log x_j}=\sum_{x_{i-1} < p \leq x_i} \frac{u_{f,j}(p) \lambda_g(p)}{p^{1/2}}.$$
Let us now define the set of cusp forms for which a given Dirichlet polynomial is smaller than a suitable threshold by
$$\mathcal{G}=\mathcal{G}_k:= \{ g \in B_{2k}: | G_{(i, I)}(g)| \leq \beta_{i}^{-3/4} \text{ for all } i=1, 2 \ldots I \}.$$
Finally, we define the exceptional sets where the given Dirichlet polynomials are large. These sets build the complement to $\mathcal{G}$ and the argument to handle these exceptional sets will be different. For $0\leq j \leq I-1$, we define
\begin{align*}\mathcal{E}(j)=\mathcal{E}_k(j):=&\big\{ g \in B_{2k}\colon |G_{(i,\ell)}(g)| \leq \beta_{i}^{-3/4} \text{ for all } 1 \leq i \leq j, \text{ for all } i \leq \ell \leq I, \\
&\text{ but } |G_{(j+1, \ell)}(g)| > \beta_{j+1}^{-3/4} \text{ for some } \ell \in\{j+1, \ldots, I\}\big\}.
\end{align*}
Note that the variance of a Dirichlet polynomial of the form $\sum_{p\leq k} \frac{\lambda_g(p)}{p^{1/2}}$ is of size $\log \log k$. Hence it is a rare event that such a Dirichlet polynomial is larger than $(\log \log k)^{3/2}$, which is roughly $\beta_i^{-3/4}$. This motivates the choice of the parameters above.
\newline
\newline
The above definitions complete the required setting for the first Dirichlet polynomial on the right-hand side of expression \eqref{Soundeasy}. To handle the second Dirichlet polynomial of expression \eqref{Soundeasy}, where the summation ranges over the primes squared, it will be convenient to introduce the following notation:
\begin{equation}\label{coefficientstwo} 
w_{f,j}(p) = \frac{(\lambda_f(p)^4 -4\lambda_f(p)^2 + 4)}{2p^{1/(\beta_j \log k)}} \frac{\log (x_j/p^2)}{\log x_j} \leq 2
\end{equation}
and
\begin{equation}\label{primessquare} P_m(g):= \sum_{2^m < p \leq 2^{m+1}} \frac{w_{f, I}(p)\lambda_g(p^2)}{p}.
\end{equation}
Furthermore, define for $m\geq 0$ the set
\begin{equation}\mathcal{P}(m) := \{g \in B_{2k}: |P_m(g)| > 2^{-m/10}, \text{ but } |P_n(g)| \leq 2^{-n/10} \text{ for all } m+1 \leq n \leq  \log k /\log 2\}.\label{badprimessquared}\end{equation}
In particular $\mathcal{P}(0)$ is the set of $g \in B_{2k}$ such that $P_n(g) < 2^{-n/10}$ for all $n$. 
The philosophy behind this definition is similar to the definition of the sets $\mathcal{E}(j)$. The variance of $P_m(g)$ is roughly of size $2^{-m}$, hence it should happen rarely that this Dirichlet polynomial is larger than $2^{-m/10}$, say. 

The following lemma, whose proof can be found in Section \ref{technicallemmas}, will be used to show that Dirichlet polynomials of the form \eqref{primessquare} are negligible.
\begin{lem}\label{primessquaredlargem}
Let $k$ be large enough and define $\mathcal{P}(m)$ as in \eqref{badprimessquared}. Suppose $(\log \log k)^2 < 2^{m+1} \leq x_I=k^{\beta_I}$, then for any $1\leq j \leq I$ we have
$$\sumh_{g\in \mathcal{P}(m)} \exp\bigg(2 \sum_{p \leq 2^{m+1}} \frac{w_{f,I}(p)\lambda_g(p^2)}{p} \bigg) \ll (\log k)^{-68}$$
\end{lem}
The next lemma allows us to replace the exponential series of a Dirichlet polynomial with a finite series. The truncation error is negligible, provided that the Dirichlet polynomial is small. In fact, this is the reason why we defined the set of Dirichlet polynomials $\mathcal{G}$.
\begin{lem}\label{truncating}
Let $\mathcal{S}\subset B_{2k}$ be a set of cusp forms and let $u(p), w(p)$ be arbitrary real numbers. Let $m, M$ be any non-negative integer and fix the real numbers $1\leq y_{i-1}<y_i$ for $1 \leq i \leq I$. Furthermore, suppose that $2^{m+1}\leq y_0$ if $M\neq0$ and
\begin{equation}\bigg| \sum_{x_{j-1} < p \leq x_j} \frac{u(p)\lambda_g(p)}{p^{1/2}}\bigg| \leq 2\beta_j^{-3/4} \label{condition}
\end{equation}
for any $1\leq j \leq I$ and $g \in \mathcal{S}$. Then we have
\begin{align*} 
&\sumh_{g\in \mathcal{S}} \exp\bigg( \sum_{x_0 < p \leq x_j} \frac{u(p)\lambda_g(p)}{p^{1/2}}\bigg)\cdot \bigg( \sum_{2^m < p \leq 2^{m+1}} \frac{w(p)\lambda_g(p^2)}{p}\bigg)^{2M} \\
\ll& \sum_{\tilde{n}} \prod_{1\leq i \leq j} \frac{1}{n_i!}\sumh_{g\in B_{2k}} \prod_{1\leq i \leq j} \bigg( \sum_{x_{i-1} < p \leq x_{i}} \frac{u(p)\lambda_g(p)}{p^{1/2}}\bigg)^{n_i} \cdot \bigg( \sum_{2^m < p \leq 2^{m+1}} \frac{w(p)\lambda_g(p^2)}{p}\bigg)^{2M}, \end{align*}
where $\tilde{n}=(n_1,\ldots, n_j)$ and each component satisfies $n_i\leq 2 \lceil 50 \beta_i^{-3/4}\rceil$.
\end{lem}
For the proof we refer the reader again to Section \ref{technicallemmas}.
\subsection{Main Contribution - Treating $\mathcal{G}$}. We are now in the position to establish our main lemmas. The following lemma resembles the computation $\mathbb{E}[\exp(X)] = \exp(\mu +\sigma^2/2)$ for a Gaussian random variable with mean $\mu$ and variance $\sigma^2$. We can think of our Dirichlet polynomial $G_{(i, I)}$ as a random variable with mean $\mu =0$ and variance $\sigma^2=\sum_p \frac{\lambda_f(p)^4}{p}$. We do not know how to integrate exponentials, so we write them as finite sum using Taylor's theorem (see Lemma \ref{truncating}). Since our Dirichlet polynomials do not take large values, we only need a few terms in the Talyor expansion, so that the resulting Dirichlet polynomials have manageable length. Having done this, we can change the order of summation, which reminds us of the linearity of expectations in a probabilistic setting. The lemmas in the previous sections then allow us to deduce the desired random behaviour. 
\begin{lem}\label{generic} We follow the notation from Section \ref{Setup}. Let $u(p)$ be any real numbers such that $|u(p)|\leq p^{1/2}$ and let $\mathcal{S} \subset B_{2k}$ such that
\begin{equation}\bigg| \sum_{x_{j-1} < p \leq x_j} \frac{u(p)\lambda_g(p)}{p^{1/2}}\bigg| \leq 2\beta_j^{-3/4} \label{condition2}
\end{equation}
for any $1\leq j \leq I$ and $g \in \mathcal{S}$. Then we have for $k$ large enough
$$\sumh_{g\in \mathcal{S}} \exp\bigg(\sum_{p\leq x_{I}} \frac{u(p)\lambda_g(p)}{p^{1/2}}\bigg) \ll \exp \bigg(\frac{1}{2}\sum_{p\leq x_{I}} \frac{u(p)^2}{p}\bigg).$$
\end{lem}
\begin{proof}
We abbreviate $$U=\sumh_{g\in \mathcal{S}} \exp\bigg(\sum_{p\leq x_{I}} \frac{u(p)\lambda_g(p)}{p^{1/2}}\bigg).$$ Using our assumption \eqref{condition2} we can directly apply Lemma \ref{truncating} with $y_i=x_i=k^{\beta_j}$, $M=0$ and see that
\begin{equation}\label{initial} U\ll \sum_{\tilde{n}} \prod_{1\leq i \leq I} \frac{1}{n_i!}\sumh_{g\in B_{2k}} \prod_{1\leq i \leq j} \bigg( \sum_{x_{i-1} < p \leq x_{i}} \frac{u(p)\lambda_g(p)}{p^{1/2}}\bigg)^{n_i}
\end{equation}
with $n_i\leq 2 \lceil 50 \beta_i^{-3/4}\rceil$.

Note that 
\begin{equation}\label{integersize}\prod_{1\leq i \leq I} x_i^{n_i} = \prod_{1\leq i \leq I} k^{\beta_i n_i}\leq \prod_{1\leq i \leq I}k^{200 \beta_i^{1/4}} \leq k^{400\beta_I^{1/4}}\leq k^2/10^4
\end{equation} for $k$ large enough. Here we used that $\beta_{i}^{1/4}$ form a geometric progression of ratio $20^{1/4}\geq 2$. We can now apply Lemma \ref{Gaussian} to the right-hand side of \eqref{initial} and see that $U$ is bounded by
$$\sum_{\tilde{n}} \prod_{1\leq i \leq I} \frac{1_{2|n_i}}{2^{n_i/2}(n_i/2)!} \bigg(\sum_{x_{i-1}<p \leq x_i }\frac{u(p)^2}{p}\bigg)^{n_i/2}+O\bigg(k^7e^{-k}\sum_{ \tilde{n}}\prod_{1\leq i \leq I} \frac{1}{n_i !}\bigg).$$
The error term is negligible since
\begin{equation}\label{errortermone}
k^7e^{-k}\sum_{\tilde{n}} \prod_{1\leq i \leq I} \frac{1}{n!}\leq k^7e^{-k} \prod_{1\leq i \leq I} \sum_{n_i \leq 200\beta_i^{-3/4}} \frac{1}{n!} \leq k^7e^{-k}e^I \leq k^8e^{-k}.
\end{equation}
Writing 
$$\sum_{\tilde{n}} \prod_{1\leq i \leq I} \frac{1_{2|n_i}}{2^{n_i/2}(n_i/2)!} \bigg(\sum_{x_{i-1}<p \leq x_i }\frac{u(p)^2}{p}\bigg)^{n_i/2}=\exp\bigg(\frac{1}{2}\sum_{p\leq x_I }\frac{u(p)^2}{p}\bigg)$$
completes the proof of the lemma.
\end{proof}
Now that we have considered the case for generic coefficients $u(p)$ let us focus on our Dirichlet polynomials of interest $G_{i,I}$, together with the Dirichlet polynomial that arises from summing over primes squared. The proof idea for the following lemma remains the same as for Lemma \ref{generic}, albeit the proof being a bit more technical.
\begin{lem}\label{newmain}Let $k$ be large enough and follow the notation from Section \ref{Setup}, then 
\begin{equation}\sumh_{g\in \mathcal{G}} \exp\bigg( \sum_{p\leq x_{I}} \frac{u_{f, I}(p) \lambda_g(p)}{p^{1/2}}\bigg)\exp\bigg( \sum_{p^2 \leq x_{I}} \frac{w_{f, I}(p)\lambda_g(p^2)}{p}\bigg)\ll\exp\bigg(\frac{1}{2} \sum_{p\leq x_{I}} \frac{u_{f,I}(p)^2}{p}\bigg).\label{estimate}\end{equation}
\end{lem}
\begin{proof}Recall the definition of the set $\mathcal{P}(m)$ in \eqref{badprimessquared}. The left-hand side of \eqref{estimate} is bounded by $$\sum_{0 \leq m \leq \log k}~\sumh_{g\in \mathcal{G} \cap \mathcal{P}(m)}\exp\bigg( \sum_{p\leq x_{I}} \frac{u_{f, I}(p) \lambda_g(p)}{p^{1/2}}\bigg)\exp\bigg( \sum_{p^2 \leq x_{I}} \frac{w_{f, I}(p)\lambda_g(p^2)}{2p}\bigg).$$
If $g\in \mathcal{P}(m)$ then clearly $$\sum_{2^{m+1}<p \leq \sqrt{x_I}}\frac{w_{f,I}(p)\lambda_g(p^2)}{p}=O(1).$$ 
In particular, if $g\in \mathcal{P}(0)$ then $\sum_{p^2 \leq x_I}\frac{w_{f,I}(p)\lambda_g(p^2)}{p}=O(1)$ and Lemma \ref{newmain} follows directly from Lemma \ref{generic} after setting $u(p)=u_{f,I}(p)$. It thus suffices to bound the quantity
$$\sum_{1 \leq m \leq \log k}~\sumh_{g\in \mathcal{G} \cap \mathcal{P}(m)}\exp\bigg( \sum_{p\leq x_{I}} \frac{u_{f, I}(p) \lambda_g(p)}{p^{1/2}}\bigg)\exp\bigg( \sum_{p \leq 2^{m+1}} \frac{w_{f, I}(p)\lambda_g(p^2)}{2p}\bigg).$$
Splitting into the sets $\mathcal{P}(m)$ enables us to show that the contribution from the primes squared part is negligible. We start by looking at the case when $m \leq (2/\log 2) \log \log \log k$. Consider the following quantity for $g \in \mathcal{P}(m)$,  which is a sum of the small primes of our Dirichlet polynomial over primes together with the primes squared Dirichlet polynomial: 
\begin{equation} \label{smallprimes}\bigg| \sum_{p \leq 2^{m+1}} \frac{u_{f, I}(p)\lambda_g(p)}{p^{1/2}}+ \sum_{p \leq 2^{m+1}} \frac{w_{f,I}(p)\lambda_g(p^2)}{2p}\bigg|.
\end{equation}
We use the triangle inequality and the Deligne bound for the Fourier coefficients, i.e. $|u_{f,I}(p)\lambda_g(p)|\leq 8$ and $|w_{f,I}(p)\lambda_g(p^2)|\leq 6$, to see that \eqref{smallprimes} is bounded by
$$\sum_{p\leq 2^{m+1}} \frac{8}{\sqrt{p}} + \sum_{p \leq 2^{m+1}} \frac{6}{p} \leq 2^{m/2}+O(1).$$
This computation is useful so that $P_m(g)$ and the prime Dirichlet polynomial are running over disjoint primes, as we have an application of Lemma \ref{Gaussian} in mind. We have
\begin{align}
T(m):=&\sumh_{g \in \mathcal{G}\cap \mathcal{P}(m)} \exp \bigg( \sum_{p \leq x_{I}}\frac{ u_{f, I}(p)\lambda_g(p)}{p^{1/2}}\bigg)\exp\left( \sum_{p \leq 2^{m+1}} \frac{w_{f, I}(p)\lambda_g(p^2)}{2p}\right) \nonumber \\
\ll & e^{2^{m/2}} \sumh_{g \in \mathcal{G}\cap \mathcal{P}(m)} \exp\bigg(\sum_{2^{m+1} < p \leq x_{I}} \frac{u_{f, I}(p)\lambda_g(p)}{p^{1/2}} \bigg)\nonumber\\
\ll & e^{2^{m/2}} \sum_{g\in \mathcal{G}} \left(2^{m/{10}}P_m(g)\right)^{2M}  \exp\bigg(\sum_{2^{m+1} < p \leq x_{I}} \frac{u_{f, I}(p)\lambda_g(p)}{p^{1/2}} \bigg) \label{beforetruncatingone},
\end{align}
where $M$ is any non-negative integer. We choose $M=\lfloor 2^{3m/4}\rfloor$ and this choice will become apparent in a calculation below. 

Now we want to replace the exponential with a finite series. Since $g \in \mathcal{G}$ and $2^m \leq (\log \log k)^2$ we have that
$$\bigg|\sum_{2^{m+1}< p \leq x_I} \frac{u_{f,I}(p)\lambda_g(p)}{p^{1/2}}\bigg| \leq |G_{(i, I)}(g)| + \bigg|\sum_{p\leq 2^{m+1}} \frac{u_{f,I}(p)\lambda_g(p)}{p^{1/2}} \bigg| \leq 2\beta_i^{-3/4}$$
and the conditions of Lemma \ref{truncating} are satisfied with $y_i:=\max\{2^{m+1}, x_i\}$. Note that since $m \leq (2/\log 2) \log \log \log k$, we have that $y_i=x_i$ for $1\leq i \leq I$ and $y_0=2^{m+1}$. An application of Lemma \ref{truncating} to \eqref{beforetruncatingone} shows that $T(m)$ is bounded by
\begin{equation}\label{beforeGaussian}
e^{2^{m/2}} 2^{mM/5}\sum_{\tilde{n}} \prod_{1\leq i \leq I} \frac{1}{n_i!} \sumh_{g\in B_{2k}} \bigg\{\prod_{1\leq i \leq I} \bigg( \sum_{x_{i-1} < p < x_{i}} \frac{u_{f, I}(p) \lambda_g(p)}{p^{1/2}} \bigg)^{n_i} \bigg\}\cdot \bigg( \sum_{2^m < p \leq 2^{m+1}} \frac{w_{f, I}(p) \lambda_g(p^2)}{p}\bigg)^{2M}.\end{equation}

The next step is to use the random behaviour of the coefficients $\lambda_g(p)$ and $\lambda_g(p^2)$ when averaged over $g\in B_{2k}$ as we did in Lemma \ref{Gaussian}.
Note that $2^{(m+1)\cdot M} \ll (\log \log k)^{2\log \log k}=k^{o(1)}$ and so we have, as already seen for inequality \eqref{integersize},
 \begin{align*}
 2^{(m+1)M}\prod_{1\leq i \leq I} x_i^{n_i}  \leq k^{o(1)}\cdot k^{400\beta_I^{1/4}}  \leq k^2/10^4
 \end{align*}
for $k$ large enough. 
We can therefore apply Lemma \ref{Gaussian} with $u(p)=u_{f, I}(p)$ and $w(p)=w_{f,I}(p)\leq 4$ so that $T(m)$ is bounded up to an error term by
$$
 e^{2^{m/2}} 2^{mM/5}\sum_{\tilde{n}} \prod_{1\leq i \leq I} \frac{1_{2|n_i}}{2^{n_i/2} (n_i/2)!} \bigg( \sum_{x_{i-1} < p \leq x_i} \frac{u_{f,I}(p)^2}{p}\bigg)^{n_i/2} \cdot \frac{(2M)!}{M!} \bigg(\frac{288}{2^m}\bigg)^M $$
with $\tilde{n}=(n_1, \ldots, n_I)$ and each component satisfies $n_i\leq 2 \lceil 50 \beta_i^{-3/4}\rceil$.
The mentioned error term is bounded as in inequality \eqref{errortermone} by
\begin{align}\label{errorterm} 
&k^7 e^{-k} \cdot e^{2^{m/2}} 2^{mM/5}\sum_{\tilde{n}} \prod_{1\leq i \leq I} \frac{1}{n_i!} 
 \ll k^{9}e^{-k}
\end{align}
and is therefore negligible.
Rearranging the Dirichlet polynomial over primes into an exponential and applying Stirlings formula, giving $(2M)!/M! \ll \big(\frac{2^{2M}M}{e}\big)^M$, we see that
\begin{equation} \label{almostlast}
T(m) \ll e^{2^{m/2}} \exp\bigg(\frac{1}{2}\sum_{2^{m+1} \leq p \leq x_I} \frac{u_{f,I}(p)^2}{p}\bigg) \cdot \bigg(\frac{2^{m/5} \cdot M \cdot 1152 \cdot 2^{-m}}{e}\bigg)^M.
\end{equation}
By our choice of $M=\lfloor 2^{3m/4}\rfloor$ we have $$\bigg(\frac{2^{m/5} \cdot 2^{3m/4} \cdot 1152 \cdot 2^{-m}}{e} \bigg)^{\lfloor 2^{3m/4}\rfloor} \ll e^{-2^{3m/4}}$$ and so
$$T(m) \ll e^{2^{m/2} - 2^{3m/4}} \cdot \exp\bigg(\frac{1}{2}\sum_{p\leq x_I} \frac{u_{f,I}(p)^2}{p}\bigg).$$
Summing $T(m)$ over $m\leq (2/\log 2) \log \log \log k$ concludes the proof of the lemma in the given range of $m$.

In the remaining case, when $(\log \log k)^2 < 2^{m+1} \leq \log k$ an application of the Cauchy-Schwarz inequality will be enough to conclude the lemma. We have
\begin{align}T(m)&= \sumh_{g\in \mathcal{G} \cap \mathcal{P}(m)}\exp\bigg( \sum_{p\leq x_{I}} \frac{u_{f, I}(p) \lambda_g(p)}{p^{1/2}}\bigg)\exp\bigg( \sum_{p \leq 2^{m+1}} \frac{w_{f, I}(p)\lambda_g(p^2)}{2p}\bigg)\nonumber\\
&\leq \bigg( \sumh_{g\in \mathcal{G}}\exp\bigg( 2\sum_{p\leq x_{I}} \frac{u_{f, I}(p) \lambda_g(p)}{p^{1/2}}\bigg)\bigg)^{1/2} \cdot \bigg(~~~\sumh_{g\in \mathcal{P}(m)} \exp\bigg( 2\sum_{p \leq 2^{m+1}} \frac{w_{f, I}(p)\lambda_g(p^2)}{p}\bigg) \bigg)^{1/2} \label{afterCauchy}
\end{align}
Using Lemma \ref{generic}, the first factor of \eqref{afterCauchy} is bounded by
$$\exp\bigg(\sum_{p\leq x_I} \frac{\lambda_f(p)^4}{p}\bigg)\ll (\log k)^{2^4}.$$
For the second part of \eqref{afterCauchy} we apply Lemma \ref{primessquaredlargem}. Combining these two bounds we see that $T(m)\ll (\log k)^{-18}$ for $m$ such that $(\log \log k)^2 < 2^{m+1} \leq \log k$. 
Summing over $m$ we have
$$\sum_{(\log \log k)^2 < 2^{m+1} \leq \log k}T(m) \ll (\log k)^{-17},$$
which is clearly negligible and so the claim of Lemma \ref{newmain} follows.

\end{proof}
\subsection{New Exceptional Set Contribution - Treating $\mathcal{E}(j)$}
In this section we treat the exceptional sets, i.e. those cusp forms where some (possibly all) parts of the Dirichlet polynomial are large. In this case we cannot apply our techniques from the last section. Although these large values cause some trouble, they are  very rare. With a Markov inequality type argument, we can indeed show that the measure of these "bad" sets is so small, that the entire contribution is negligible. Unsurprisingly, the argument will remind us of the treatment of the primes squared part in Lemma \ref{newmain}. 

Recall that we are now interested in the set of cuspforms, where the corresponding Dirichlet polynomial might get large.
For  $0\leq j \leq I-1$, we defined
\begin{align*}\mathcal{E}(j)=\mathcal{E}_k(j):=&\big\{ g \in B_{2k}\colon |G_{(i,\ell)}(g)| \leq \beta_{i}^{-3/4} \text{ for all } 1 \leq i \leq j, \text{ for all } i \leq \ell \leq I, \\
&\text{ but } |G_{(j+1, \ell)}(g)| > \beta_{j+1}^{-3/4} \text{ for some } \ell \in\{j+1, \ldots, I\}\big\}.
\end{align*}
\begin{lem}\label{exceptional}
For $k$ large enough and following the notation in Section \ref{Setup}, we have
$$\meas\{\mathcal{E}(0)\}=\sumh_{g\in \mathcal{E}(0)}1 \ll e^{-(\log \log k)^2/C}$$
with $C=2^5\cdot 10/e$. Moreover, for any $1 \leq j \leq I-1$ we have that
$$\sumh_{g \in \mathcal{E}(j)}\exp \bigg( \sum_{p \leq x_j} \frac{u_{f,j}(p)\lambda_g(p)}{p^{1/2}}\bigg) \exp\bigg(\sum_{p^2 \leq x_j} \frac{w_{f,j}(p)\lambda_g(p^2)}{p} \bigg) \ll \exp\bigg(\frac{1}{2} \sum_{p \leq x_j}\frac{u_{f,j}(p)^2}{p} \bigg) e^{(4C \beta_{j+1})^{-1} \log \beta_{j+1}}$$
\end{lem}
\begin{proof}
We treat the primes squared part as in the proof of Lemma \ref{newmain}. By the exact same reduction as in Lemma \ref{newmain} it suffices to control
\begin{align}S(m):=&\sumh_{g \in \mathcal{E}(j) \cap \mathcal{P}(m)}\exp \bigg( \sum_{p \leq x_j} \frac{u_{f,j}(p)\lambda_g(p)}{p^{1/2}}\bigg) \exp\bigg(\sum_{p^2 \leq x_j} \frac{w_{f,j}(p)\lambda_g(p^2)}{p} \bigg) \nonumber\\
\ll &e^{2^{m/2}}  \sumh_{g\in \mathcal{E}(j)} \exp \bigg( \sum_{2^{m+1} < p\leq x_j} \frac{u_{f,j}(p) \lambda_g(p)}{p^{1/2}}\bigg) \cdot (2^{m/10} P_m(g))^{2M}
\end{align}
for $m \leq (2/\log 2) \log \log \log k$.
By the definition of the set $\mathcal{E}(j)$ and Markov's inequality $S(m)$ is bounded by
\begin{align}
&e^{2^{m/2}} \sum_{\ell=j+1}^I \sumh_{\substack{g \in B_{2k}: |G_{(i,j)}(g)| \leq \beta_i^{-3/4} \forall 1 \leq i\leq j,\\  |G_{j+1, \ell}(g)|>\beta_{j+1}^{-3/4}}}  \exp \bigg(\sum_{2^{m+1} \leq p \leq x_j} \frac{u_{f,j}(p)\lambda_g(p)}{p^{1/2}}\bigg)\cdot (2^{m/10} P_m(g))^{2M}\\
\leq & e^{2^{m/2}} \sum_{\ell=j+1}^I \sumh_{\substack{g \in B_{2k}: |G_{(i,j)}(g)| \leq \beta_i^{-3/4} \\ \forall 1 \leq i\leq j}}  \exp \bigg(\sum_{2^{m+1} \leq p \leq x_j} \frac{u_{f,j}(p)\lambda_g(p)}{p^{1/2}}\bigg) \big(\beta_{j+1}^{3/4} G_{(j+1, \ell)}(g)\big)^{2L} \left(2^{m/10}P_m(g)\right)^{2M} \label{exceptionalbound}
\end{align}
where $L$ is any non-negative integer, which we choose to be $L= \lfloor(C \beta_{j+1})^{-1}\rfloor$, with $C= 2^5 \cdot 10 / e$. Now we are again in the position to truncate the exponential and proceed as in Lemma \ref{newmain}, more precisely by Lemma \ref{truncating} we get that $S(m)$ is bounded by
$$e^{2^{m/2}}2^{mM/5}\beta_{j+1}^{3L/2}\sum_{\ell=j+1}^I \sum_{\tilde{n}} \prod_{1\leq i \leq j}\frac{1}{n_i!} \sumh_{g\in B_{2k}} \prod_{1\leq i \leq j}\bigg(\sum_{x_{i-1} <p \leq x_i} \frac{u_{f,j}(p)\lambda_g(p)}{p^{1/2}}\bigg)\cdot G_{(j+1, \ell)}^{2L}\cdot P_m(g)^{2M}$$
with $\tilde{n}=(n_1, \ldots n_j)$, and each component satisfies $n_i\leq 2 \lceil 50 \beta_i^{-3/4}\rceil$. Again we use Lemma \ref{Gaussian} to capture the random behaviour of the coefficients $\lambda_g(p)$ and $\lambda_g(p)^2$. This lemma is applicable since
$$2^{(m+1)2M}\cdot x_{j+1}^{2L}\prod_{1\leq i \leq j} x_i^{n_i}  \leq  k^{o(1)}\cdot k^{2/C}\prod_{1\leq i \leq I} k^{100 {\beta_i}^{1/4}}  \leq k^2/10^4.$$
Then the main term of $S(m)$ is bounded by
\begin{align*}e^{2^{m/2}}2^{mM/5}\beta_{j+1}^{3L/2} \sum_{\ell=j+1}^I &\sum_{\tilde{n}} \bigg\{\prod_{1\leq i \leq j} \frac{1_{2|n_i}}{2^{n_i/2}(n_i/2)!} \bigg( \sum_{x_{i-1} < p \leq x_i} \frac{u_{f,j}(p)^2}{p}\bigg)^{n_i/2} \bigg\} \cdot  \\ &\cdot \frac{(2L)!}{L!} \bigg( \sum_{x_{j} < p \leq x_{j+1}} \frac{u_{f, j+1}(p)^2}{p}\bigg)^L \cdot \frac{(2M)!}{M!} \bigg(\frac{288}{2^m}\bigg)^M.
\end{align*}
As in Lemma \ref{newmain} we write this in terms of an exponential and we use the bound $u_{f,j+1}(p) \leq \lambda_f(p)^2$, so that $S(m)$ is controlled by
$$e^{2^{m/2}}2^{mM/5}\beta_{j+1}^{3L/2} (I-j) \exp\bigg(\frac{1}{2} \sum_{p \leq x_j} \frac{u_{f,j}(p)^4}{p} \bigg) \cdot \frac{(2L)!}{2^L L!} \bigg(\sum_{x_j < p \leq x_{j+1}} \frac{\lambda_f(p)^4}{p}\bigg)^L \cdot \frac{(2M)!}{M!} \left(\frac{288}{2^m}\right)^M.$$
The error term arising from Lemma \ref{Gaussian} is again negligible by the same computation as in \eqref{errorterm}.
Together with a Stirling estimate this computation yields
\begin{equation}S(m) \ll e^{2^{m/2}}(I-j) \exp \bigg( \frac{1}{2} \sum_{p\leq x_j} \frac{u_{f,j}(p)^2}{p}\bigg) \cdot \bigg(\frac{\beta_{j+1}^{3/2} \cdot 2L}{e} \sum_{x_j< p \leq x_{j+1}} \frac{\lambda_f(p)^4}{p}\bigg)^L \cdot \bigg(\frac{2^{m/5} \cdot M \cdot 1152 \cdot 2^{-m}}{e}\bigg)^M \label{exceptionalfinal}
\end{equation}

In the case $1 \leq j \leq I-1$ we have by the definition of $\beta_j$ and $I$, that 
$$I-j =\frac{\log(\beta_I/\beta_j)}{\log 20} \leq \frac{\log (1/\beta_j)}{\log 20}$$ and
$$\sum_{k^{\beta_j} < p \leq k^{\beta_{j+1}}} \frac{\lambda_f(p)^4}{p} \leq 2^4 (\log \beta_{j+1} - \log \beta_j + o(1)) = 2^4 (\log 20 + o(1)) \leq 2^4 \cdot 10.$$
Consequently, 
\begin{equation}\label{exceptionalfinal}(I -j) \cdot \left(\frac{\beta_{j+1}^{3/2}\cdot 2L}{e} \sum_{k^{\beta_j} < p \leq k^{\beta_{j+1}}} \frac{\lambda_f(p)^4}{p} \right)^{L}\leq \frac{\log(1/\beta_j)}{\log 20} (\beta_{j+1}^{1/2})^{1/(C \cdot \beta_{j+1})} 
\end{equation}
The right-hand side of inequality \eqref{exceptionalfinal} is bounded by $$e^{(4 C\cdot \beta_{j+1})^{-1} \log \beta_{j+1}},$$ 
which is small since $\beta_{j+1}\leq \beta_{I} \leq 20e^{-10^5}$. Summing over $m$ as we did in Lemma \ref{newmain} shows that \eqref{exceptionalfinal} is bounded by
$$\exp\bigg(\frac{1}{2}\sum_{p \leq x_j} \frac{u_{f,j}(p)^2}{p}\bigg) e^{(4C \beta_{j+1})^{-1} \log \beta_{j+1}}.$$
The remaining case when $m\geq (2/\log 2) \log \log \log k$ is  negligible compared to the main term. As in the proof of Lemma \ref{newmain} this can be seen by an application of the Cauchy-Schwarz inequality and Lemma \ref{primessquaredlargem}. This finishes the proof of the lemma for the cases $1\leq j \leq I$.

It remains to show the first assertion of the lemma, namely
$$\sumh_{g\in \mathcal{E}(0)} 1 \ll e^{-\log \log k^2/C}.$$
Note that from the definition of $I$ we see that $I \leq \log \log \log k$. Moreover,
$$\beta_0=0,~~\beta_1=\frac{1}{(\log \log k)^2},~~ \sum_{p \leq k^{1/(\log \log k)^2}} \frac{\lambda_f(p)^4}{p} \leq 2^4 \log \log k.$$

Following the argument from before for $1\leq j \leq I$ without the exponential factors we see that
\begin{align*}\sumh_{g\in \mathcal{E}(0)} 1 &\ll I \cdot \bigg(\frac{\beta_1^{3/2} \cdot 2L}{e} \sum_{p \leq k^{\beta_1}} \frac{\lambda_f(p)^4}{p}\bigg) \\
&\ll \log \log \log k \cdot \bigg(\frac{\beta_1^{3/2} \cdot 2L}{e} \cdot 2^4 \log \log k \bigg)^L \\
&\ll e^{-(\log \log k)^2/C}
\end{align*}
by our choice of $L$ and $C$. This finishes the proof of the entire lemma.

\end{proof}
\subsection{Technical Lemmas} \label{technicallemmas}
In this section we quickly prove certain technical statements that were used in the section before. We also gather some additional technical lemmas that are needed in the final proof of Theorem \ref{main}.

\begin{proof}[Proof of Lemma \ref{primessquaredlargem}]
Since $|\lambda_g(p^2)|\leq 3$ and $|w_{f,I}(p)|^2\leq 2$ we have that
$$2\sum_{p\leq 2^{m+1}} \frac{b_{f,I}(p)\lambda_g(p^2)}{p} \ll 12 \log \log 2^{m+1}.$$
An application of Markov's inequality yields
\begin{align}
B(m):=\sumh_{g\in \mathcal{P}(m)} \exp\bigg( 2\sum_{p \leq 2^{m+1}} \frac{w_{f,I}(p)\lambda_g(p^2)}{p} \bigg) &\ll (\log 2^{m+1})^{12} \sumh_{g\in \mathcal{P}(m)} 1 \nonumber
\\ &\leq (\log 2^{m+1})^{12} \sumh_{g\in B_{2k}} (2^{m/10} P_m(g))^{2M} \label{squaredmoments}
\end{align}
for any non-negative integer $M$. We apply Lemma \ref{Gaussian} (with $n_i=0$ for $1\leq i \leq I$) to evaluate the above moment and get 
\begin{equation}
B(m) \ll (\log 2^{m+1})^{12}\cdot \frac{(2M)!}{M!} \bigg(\frac{72C^2 \cdot  2^{m/5}}{2^m} \bigg)^M, \label{finalestimatesquared}
\end{equation} 
provided that $2^{(m+1)2M}\leq k^2/10^4$. 
We first investigate the case when $\log k \leq 2^{m+1} \leq \sqrt{x_I}$. In this range we have 
$$2^{(m+1)2M} \leq k^{\beta_I M} \leq k^{20 e^{-10000} M}$$
and so $M=100$ is certainly admissible. 
By our choice of $M$ and taking into account the size of $2^m$ we see that 
$$B(m) \ll (\log k)^{12} 2^{-400m/5} \ll (\log k)^{12}\cdot (\log k)^{-80}=(\log k)^{-68}.$$

Next we consider the case $(\log \log k)^2 \leq 2^{m+1} \leq \log k$. Since the primes $p\leq 2^{m+1}$ are smaller in size we can afford to take higher moments. We pick $M=\lfloor 2^{3m/4}\rfloor$ so that $2^{(m+1)2M} \leq (\log k)^{(\log k)^{3/4}} \ll k^{o(1)} \leq k^2/10^4$. Together with the Stirling bound $(2M)!/M! \ll (4M/e)^M$ we see that $B(m)$ is bounded by
\begin{align*}(\log 2^{m+1})^{12}  \bigg(\frac{M \cdot 4608 \cdot 2^{m/5}}{e \cdot 2^{m}} \bigg)^M &\ll (\log \log k)^{15} e^{-2^{3m/4}}\\ &\ll (\log \log k)^{12} \exp(- (\log \log k)^{3/2})\\
&\ll (\log k)^{-68}.
\end{align*}
We used that $2^{-m/20} \cdot 4608 \leq 1$, if $k$ is sufficiently large and therefore also $m$ is sufficiently large. This completes the proof of the lemma.
\end{proof}
The following lemma, due to Radziwi\l\l~ and Soundararajan \cite[Lemma 1]{radziwillMomentsDistributionCentral2014}, will be helpful  in the process of replacing the exponential series with a finite sum.
\begin{lem} \label{exptrunc}
Let $\ell$ be a non-negative even integer, and $x$ a real number. Define $$E_{\ell}(x)=\sum_{j=0}^\ell \frac{x^j}{j!}.$$
Then $E_{\ell}(x)$ is positive and for any $x\leq 0$ we have $E_{\ell}(x) \geq e^x$. Moreover, if $x \leq \ell/e^2$, then we have
$$\exp(x) \leq \exp\big(O(e^{-\ell})\big) E_{\ell}(x).$$
\end{lem}
\begin{proof}[Proof of Lemma \ref{truncating}]
Our goal is to truncate the exponential series $\exp(x)$ and replace it with a finite series up to $\ell$. During this process we incur a negligible error term, provided that $x$ is smaller than $\ell$ ( see for example Lemma \ref{exptrunc}). This is the case for our Dirichlet polynomials by assumption \eqref{condition}. With $\ell= 2 \lceil 50 \beta_i^{-3/4}\rceil$ we have
\begin{align}
& \sumh_{g\in \mathcal{S}} \exp\bigg( \sum_{x_0 < p \leq x_j} \frac{u(p)\lambda_g(p)}{p^{1/2}}\bigg)\cdot \bigg( \sum_{2^m < p \leq 2^{m+1}} \frac{w(p)\lambda_g(p^2)}{p}\bigg)^{2M} \nonumber\\
=& \sumh_{g\in \mathcal{S}} \prod_{1\leq i \leq j} \exp\bigg( \sum_{x_{i-1} < p \leq x_i} \frac{u(p)\lambda_g(p)}{p^{1/2}} \bigg)\cdot\bigg( \sum_{2^m < p \leq 2^{m+1}} \frac{w(p)\lambda_g(p^2)}{p}\bigg)^{2M}\nonumber\\
\leq &\sumh_{g \in \mathcal{S}} \prod_{1 \leq i \leq j} \exp\left(O(e^{-100\beta_i^{-3/4}})\right) \sum_{0 \leq n \leq \ell} \frac{1}{n!} \bigg( \sum_{x_{i-1} < p \leq x_i} \frac{u(p)\lambda_g(p)}{p^{1/2}} \bigg)^n  \bigg( \sum_{2^m < p \leq 2^{m+1}} \frac{w(p)\lambda_g(p^2)}{p}\bigg)^{2M}\nonumber\\
\ll &\sumh_{g \in \mathcal{S}} \prod_{1 \leq i \leq j}  \sum_{0 \leq n \leq \ell} \frac{1}{n!}\bigg( \sum_{x_{i-1} < p \leq x_i} \frac{u(p)\lambda_g(p)}{p^{1/2}} \bigg)^n  \bigg( \sum_{2^m < p \leq 2^{m+1}} \frac{w(p)\lambda_g(p^2)}{p}\bigg)^{2M}\ \label{beforeexpansion}
\end{align}
In the third equality we used assumption \eqref{condition} and Lemma \ref{exptrunc}. Note that $\sum_{0\leq n \leq \ell}\frac{x^n}{n!} \geq 0$ for every $x$, as $\ell$ is even. Using this positivity, we replace the sum over the restricted set $\sum_{g\in \mathcal{S}}^h$ with the full sum $\sum_{g\in B_{2k}}^h$. Additionally, we expand the product over $i$ and so \eqref{beforeexpansion} is equal to
$$\sum_{\tilde{n}} \prod_{1\leq i \leq j} \frac{1}{n_i!} \sumh_{g\in B_{2k}} \bigg( \sum_{x_{i-1} < p \leq x_i} \frac{u(p)\lambda_g(p)}{p^{1/2}} \bigg)^{n_i}  \bigg( \sum_{2^m < p \leq 2^{m+1}} \frac{w(p)\lambda_g(p^2)}{p}\bigg)^{2M}$$
with $\tilde{n}=(n_1, \ldots, n_I)$ where each component satisfies $n_i \leq \ell$. This concludes the proof.
\end{proof}
For technical reason in the proof of Theorem \ref{main} we will need the following lemma
\begin{lem}\label{techn} For any $1 \leq i \leq I$ and write $x_i=k^{\beta_i}$, then we have
$$\exp\bigg(\frac{1}{2} \sum_{p \leq x_i} \frac{\lambda_f(p)^4}{p^{1+2/\log x_i}} \frac{\log^2(x_i/p)}{\log^2 x_i}\bigg)\cdot \exp\bigg(-\frac{1}{2}\sum_{p \leq \sqrt{x_i}}\frac{\lambda_f(p)^4}{p^{1+2/\log x_i}} \frac{\log(x_i/p^2)}{\log x_i}\bigg)=O(1)$$
\end{lem}
\begin{proof}
At first we investigate the primes up to $\sqrt{x_i}$. We want to estimate
$$\exp\bigg(\frac{1}{2} \sum_{p \leq \sqrt{x_i}} \frac{\lambda_f(p)^4}{p^{1+2/\log x_i}} \frac{\log^2(x_i/p)}{\log^2 x_i}-\frac{1}{2}\sum_{p \leq \sqrt{x_i}}\frac{\lambda_f(p)^4}{p^{1+2/\log x_i}} \frac{\log(x_i/p^2)}{\log x_i} \bigg)$$
After expanding the smoothing of $\log(x_i/p)$ for both sums, we see that the only contribution that is left comes from
$$\exp\bigg(\frac{1}{2} \sum_{p \leq \sqrt{x_i}} \frac{\lambda_f(p)^4}{p^{1+2/\log x_i}} \frac{(\log p)^2}{(\log x_i)^2}\bigg).$$
We bound $\frac{\lambda_f(p)^4}{p^{2/\log x_i}} \frac{\log p}{\log x_i}$ trivially by a constant (here we use the Deligne bound for the Fourier coefficients) and see that
$$\exp\bigg(\frac{1}{2} \sum_{p \leq \sqrt{x_i}} \frac{\log p}{p} \frac{1}{\log x_i}\bigg)=O(1).$$
It remains to show that 
\begin{equation}\exp\bigg(\frac{1}{2} \sum_{\sqrt{x_i} < p \leq x_i} \frac{\lambda_f(p)^4}{p^{1+2/\log x_i}} \frac{\log^2(x_i/p)}{\log^2 x_i}\bigg) \label{technicalstep}
\end{equation} is bounded. Putting absolute values, and using again the Deligne bound, expression \eqref{technicalstep} is controlled by
$$\exp\bigg( \sum_{\sqrt{x_i} < p \leq x_i} \frac{2^3}{p}\bigg) \ll \exp\left( \log \log x_i - \log \log \sqrt{x_i}\right) =O(1).$$
Hence, the lemma follows.
\end{proof}
\begin{lem} \label{sym}Assume the Riemann Hypothesis for $L(s, \sym^2f)$. For any $1 \leq i \leq I$ we have
$$\frac{1}{L(1,\sym^2f)^2}\cdot \exp\left(\sum_{p\leq \sqrt{k^{\beta_i}}} \frac{2\lambda_f(p)^2-2}{p}\right)=O(1)$$
\end{lem}
\begin{proof}
This is a small modification of Lemma 2 in \cite{holowinskyMassEquidistributionHecke2010a}. Instead of the zero free region we use the Riemann Hypothesis for $L(s, \sym^2f)$ to bound the contribution of the zeros. 
\end{proof}
The next lemma is a crude bound for the second moment of our degree eight $L$-function. The ideas are from \cite{soundararajanMomentsRiemannZeta2009a} and adapted to our context.
\begin{lem}\label{weakSound} Let $f$ and $g$ be Hecke cusp forms of even weight $k$ and $2k$ respectively for the full modular group. Assuming Riemann Hypothesis for $L(1/2, f\times f \times g)$\label{weakSound}
$$\sumh_{g \in B_{2k}} L(1/2, f\times f\times g)^{2} \ll (\log k)^{10^{30}}$$
\end{lem}

\begin{proof}
Define $S(g, V):= \{g \in B_{2k}: \log L(1/2, f \times f \times g) \geq V\}.$
Notice that $$\sumh_{g\in B_{2k}} L(1/2, f \times f \times g)^{2} = \int_{-\infty}^\infty e^{2 V} \meas(S(g,V)) dV$$
It suffices to investigate
\begin{equation}\label{int}
\int_{10^{30} \log \log k}^\infty e^{2V} \meas(S(g,V)) dV
\end{equation}
 as otherwise we trivially have the desired result.
\newline
\newline
From Lemma \ref{Sound} we have for any $x \geq 2$ that
\begin{align*}\log L(1/2, f \times f \times g) \leq &\sum_{p \leq x} \frac{\lambda_f(p)^2\lambda_g(p)}{p^{\frac{1}{2}+\frac{1}{\log x}}} \frac{\log (x/p)}{\log x} \\&
+ \sum_{p\leq \sqrt{x}} \frac{(\lambda_f(p)^4-4\lambda_f(p)^2+4)(\lambda_g(p^2)-1)}{2p^{1+\frac{2}{\log x}}}\frac{\log (x/p^2)}{\log x} + \frac{\log k^6}{\log x} + O(1)\\
\leq& \sum_{p \leq x} \frac{\lambda_f(p)^2\lambda_g(p)}{p^{\frac{1}{2}+\frac{1}{\log x}}} \frac{\log (x/p)}{\log x} + 6 \log \log x + \frac{6 \log k}{\log x} + O(1).
\end{align*}
Here we used that $|\lambda_f(p)|\leq 2$ and $|\lambda_g(p)|\leq 3$. If we pick $x=k^{16/V}$, and notice that $6 \log \log k \leq ( 6/10^{30}) V $, then 
$$\log L(1/2, f \times f \times g) \leq  \sum_{p \leq x} \frac{\lambda_f(p)^2\lambda_g(p)}{p^{\frac{1}{2}+\frac{1}{\log x}}} \frac{\log (x/p)}{\log x} + \frac{3V}{4} + O(1)$$
Hence if $g \in S(g, V)$ then $$\sum_{p \leq x} \frac{\lambda_f(p)^2\lambda_g(p)}{p^{\frac{1}{2}+\frac{1}{\log x}}} \frac{\log (x/p)}{\log x} \geq \frac{V}{4}.$$
By Markov's inequality, we have for any non-negative integer $n$
$$\meas(S(V,g)) \leq \frac{4^{2n}}{V^{2n}} \sumh_{g\in B_{2k}} \bigg(\sum_{p \leq x} \frac{\lambda_f(p)^2\lambda_g(p)}{p^{\frac{1}{2}+\frac{1}{\log x}}} \frac{\log (x/p)}{\log x}\bigg)^{2n}.$$
By Lemma \ref{Gaussian} this is bounded by
\begin{equation} \frac{4^{2n}}{V^{2n}} \cdot \frac{(2n)!}{2^{2n}n!} \bigg(\sum_{p\leq x} \frac{\lambda_f(p)^4}{p}\bigg)^n
\label{weaksoundmoment}
\end{equation}
provided that $x^{2n} \leq k^2/10^4$. From our choice of $x$ we see that $n=\floor{V/20}$ is admissible. 
By Stirling and the Deligne bound quantity \eqref{weaksoundmoment} is controlled by 
$$\left(\frac{ 2^8 n \log \log k}{V^2 \cdot e}\right)^n.$$
This in turn is bounded by 
$$\left(\frac{2^8}{20 \cdot 10^{30} e}\right)^n \ll e^{-3V}$$
by our choice of $n$ and the lower bound $V \geq 10^{30} \log \log k$.
We see that the contribution of the integral in \eqref{int} is negligible and consequently the result follows.
\end{proof}
\subsection{Proof of Theorem \ref{main}}
\begin{proof}[Proof of Theorem \ref{main}]
Note that $$\{ g \in B_{2k}\} = \mathcal{G} \cup \bigcup_{j=0}^{I-1}\mathcal{E}(j),$$
hence our goal is to show that 
\begin{equation}
\sumh_{g\in \mathcal{G}} \frac{L(1/2, f\times f \times g)}{L(1,\sym^2f)^2} + \sum_{j=0}^{I-1}~~ \sumh_{g\in \mathcal{E}(j)}\frac{L(1/2, f\times f \times g)}{L(1,\sym^2f)^2}=O(1). \label{goal}
\end{equation}
At first we approximate the $L$-functions with Dirichlet polynomials. Lemma \ref{Sound} gives for $x=x_I=k^{\beta_{I}}$
\begin{align*}
\log L(1/2, f \times f \times g) \leq &\sum_{p \leq x_I} \frac{\lambda_f(p)^2\lambda_g(p)}{p^{1/2+1/\log x_I}} \frac{\log (x_I/p)}{\log x_I} \\&
+ \sum_{p\leq \sqrt{x_I}} \frac{(\lambda_f(p)^4-4\lambda_f(p)^2+4)(\lambda_g(p^2)-1)}{2p^{1+2/\log x_I}}\frac{\log (x_I/p^2)}{\log x_I} +\frac{6}{\beta_I} + O(1)
\end{align*}
Consequently, the first sum in \eqref{goal} is bounded by 
\begin{align} \label{goodcase} e^{6/\beta_{I} }&\sumh_{g\in \mathcal{G}} \exp\bigg( \sum_{p \leq x_I} \frac{\lambda_f(p)^2\lambda_g(p)}{p^{1/2+1/\log x_I}} \frac{\log (x_I/p)}{\log x_I} \bigg) \cdot \\
 \cdot &\exp\bigg(\sum_{p\leq \sqrt{x_I}} \frac{(\lambda_f(p)^4-4\lambda_f(p)^2+4)\lambda_g(p^2)}{2p^{1+2/\log x_I}}\frac{\log (x_I/p^2)}{\log x_I} \bigg) \cdot \nonumber \\
\cdot &\exp\bigg(- \sum_{p\leq \sqrt{x_I}} \frac{(\lambda_f(p)^4-4\lambda_f(p)^2+4)}{2p^{1+2/\log x_I}}\frac{\log (x_I/p^2)}{\log x_I} \bigg) \cdot \frac{1}{L(1, \sym^2f)^2} \nonumber
\end{align}
By Lemma \ref{newmain} the contribution of the first two exponential sums is bounded by
$$\exp\bigg(\frac{1}{2} \sum_{p\leq x_I} \frac{\lambda_f(p)^4}{p^{1+2\log x_I}} \frac{\log^2(x_I/p)}{(\log x_I)^2}\bigg)$$
The last exponential factor of \eqref{goodcase} can be written as
$$\exp\bigg(- \frac{1}{2} \sum_{p\leq \sqrt{x_I}} \frac{\lambda_f(p)^4}{p^{1+2/\log x_I}}\frac{\log (x_I/p^2)}{\log x_I} \bigg) \cdot \exp\bigg( \sum_{p\leq \sqrt{x_I}} \frac{(2\lambda_f(p)-2)}{p^{1+2/\log x_I}}\frac{\log (x_I/p^2)}{\log x_I} \bigg).$$
Therefore \eqref{goodcase} can be bounded by
\begin{align*}
e^{6/\beta_I} &\exp\left(\frac{1}{2} \sum_{p\leq x_I} \frac{\lambda_f(p)^4}{p^{1+2\log x_I}} \frac{\log^2(x_I/p)}{(\log x_I)^2}\right) \cdot \exp\bigg(- \frac{1}{2} \sum_{p\leq \sqrt{x_I}} \frac{\lambda_f(p)^4}{p^{1+2/\log x_I}}\frac{\log (x_I/p^2)}{\log x_I} \bigg) \\
&\cdot \exp\bigg( \sum_{p\leq \sqrt{x_I}} \frac{(2\lambda_f(p)-2)}{p^{1+2/\log x_I}}\frac{\log (x_I/p^2)}{\log x_I} \bigg) \cdot \frac{1}{L(1, \sym^2f)^2}
\end{align*}
Since $\beta_I$ is bounded, Lemma \ref{techn} and Lemma \ref{sym} show that \eqref{goodcase} is of size $O(1)$. 

We now treat the exceptional sets from the second term in \eqref{goal}. We begin, as before, by approximating the $L$-function with Dirichlet polynomials. Lemma \ref{Sound} with $x=x_j=k^{\beta_j}$ shows that
$$ \sumh_{g\in \mathcal{E}(j)}\frac{L(1/2, f\times f \times g)}{L(1,\sym^2f)^2}$$ is bounded by
\begin{align}
e^{6/\beta_j} \cdot &\sumh_{g \in \mathcal{E}(j)}\exp \bigg( \sum_{p \leq x_j} \frac{\lambda_f(p)^2\lambda_g(p)}{p^{1/2+1/\log x_j}}\frac{\log (x_j/p)}{\log x_j}\bigg) \cdot \label{badcase}  \\
 \cdot &\exp\bigg(\sum_{p\leq \sqrt{x_j}} \frac{(\lambda_f(p)^4-4\lambda_f(p)^2+4)\lambda_g(p^2)}{2p^{1+2/\log x_j}}\frac{\log (x_j/p^2)}{\log x_j} \bigg) \cdot \nonumber \\
\cdot &\exp\bigg(- \sum_{p\leq \sqrt{x_j}} \frac{\lambda_f(p)^4-4\lambda_f(p)^2+4}{2p^{1+2/\log x_j}}\frac{\log (x_j/p^2)}{\log x_j} \bigg) \cdot \frac{1}{L(1, \sym^2f)^2} \nonumber
\end{align}
for $1\leq j \leq I-1$. By Lemma \ref{exceptional} the sum of the first two exponentials in \eqref{badcase} is bounded by
$$ \exp\left(\frac{1}{2} \sum_{p \leq x_j} \frac{\lambda_f(p)^4}{p^{1+2/\log x_j}}\frac{\log^2( x_j/p)}{(\log x_j)^2} \right) e^{(4C \beta_{j+1})^{-1} \log \beta_{j+1}}.$$
with $C=2^5\cdot 10/e$. Similarly as before, we use Lemma \ref{techn} and Lemma \ref{sym} to show that expression \eqref{badcase} is bounded by $$e^{6/\beta_j} \cdot e^{(4C \beta_{j+1})^{-1} \log \beta_{j+1}} = e^{6/\beta_j + \log (\beta_{j+1})/(80 C \beta_j)}.$$
Moreover, since  $\beta_{j+1} \leq \beta_{I} \leq 20e^{-10^5}$ we have
$$ e^{6/\beta_j + \log (\beta_{j+1})/(80 C \beta_j)} \leq e^{6/\beta_j - 10/\beta_j}=e^{-4/\beta_j}.$$
The sum over these values from $1 \leq j \leq I-1$ remains bounded and so we conclude the proof of the theorem for these exceptional sets. 
\newline
\newline
The only case that is left is when $j=0$. In that scenario, we win because the measure of $\mathcal{E}(0)$ is tiny. By Cauchy-Schwarz we have 
\begin{equation}\sumh_{g\in \mathcal{E}(0)}\frac{L(1/2, f\times f \times g)}{L(1,\sym^2f)^2} \leq \bigg(\sumh_{~g\in \mathcal{E}(0)}1\bigg)^{1/2} \cdot \bigg(~~ \sumh_{g\in B_{2k}}\frac{L(1/2, f\times f \times g)^2}{L(1,\sym^2f)^4} \bigg)^{1/2}. \label{lastcase}
\end{equation}
Note that $L(1, \sym^2f)^{-1} \ll \log k$ (see \cite{hoffsteinCoefficientsMaassForms1994} and \cite{goldfeldAppendixEffectiveZeroFree1994a}). Lemma \ref{exceptional} and Lemma \ref{weakSound} show that the right hand side of \eqref{lastcase} is bounded by
$$e^{-(\log \log k)^2/(2C)} \cdot (\log k)^{(10^{30}+4)/2}.$$
For $k$ large enough this is clearly bounded and therefore the theorem follows.
\end{proof}
\section{Acknowledgements}
The author would like to thank Dimitris Koukoulopoulos and Maksym Radziwi\l\l ~for their continuous support during this project and suggestions that significantly improved the exposition of this paper. Moreover, the author would like to thank Andrei Shubin for many helpful discussions. Part of this work was conducted at CalTech University, the author is grateful for their hospitality.
\bibliographystyle{alpha}
\bibliography{Bib2}

\begin{thebibliography}{BKY13}

\bibitem[Ber77]{berryRegularIrregularSemiclassical1977}
M.~V. Berry.
\newblock Regular and irregular semiclassical wavefunctions.
\newblock {\em J. Phys. A}, 10(12):2083--2091, 1977.

\bibitem[BK17]{buttcaneFourthMomentHecke2017}
J.~Buttcane and R.~Khan.
\newblock On the fourth moment of {H}ecke-{M}aass forms and the random wave
  conjecture.
\newblock {\em Compos. Math.}, 153(7):1479--1511, 2017.

\bibitem[BKY13]{blomerDistributionMassHolomorphic2013}
V.~Blomer, R.~Khan, and M.~Young.
\newblock Distribution of mass of holomorphic cusp forms.
\newblock {\em Duke Math. J.}, 162(14):2609--2644, 2013.

\bibitem[Cha09]{chandeeExplicitUpperBounds2009}
V.~Chandee.
\newblock Explicit upper bounds for {$L$}-functions on the critical line.
\newblock {\em Proc. Amer. Math. Soc.}, 137(12):4049--4063, 2009.

\bibitem[Gar87]{garrettDecompositionEisensteinSeries1987}
P.~Garrett.
\newblock Decomposition of {E}isenstein series: {R}ankin triple products.
\newblock {\em Ann. of Math. (2)}, 125(2):209--235, 1987.

\bibitem[GHL94]{goldfeldAppendixEffectiveZeroFree1994a}
D.~Goldfeld, J.~Hoffstein, and D.~Lieman.
\newblock Appendix: {{An Effective Zero}}-{{Free Region}}.
\newblock {\em Ann. of Math. (2)}, 140(1):177--181, 1994.

\bibitem[GS12]{ghoshRealZerosHolomorphic2011b}
A.~Ghosh and P.~Sarnak.
\newblock Real zeros of holomorphic {H}ecke cusp forms.
\newblock {\em J. Eur. Math. Soc. (JEMS)}, 14(2):465--487, 2012.

\bibitem[Har13]{harperSharpConditionalBounds2013}
A.~J. Harper.
\newblock Sharp conditional bounds for moments of the {{Riemann}} zeta
  function.
\newblock {\em arXiv:1305.4618 [math]}, 2013.

\bibitem[HL94]{hoffsteinCoefficientsMaassForms1994}
J.~Hoffstein and P.~Lockhart.
\newblock Coefficients of {M}aass forms and the {S}iegel zero.
\newblock {\em Ann. of Math. (2)}, 140(1):161--181, 1994.

\bibitem[HR92]{hejhalTopographyMaassWaveforms1992}
D.~A. Hejhal and B.~N. Rackner.
\newblock On the topography of {M}aass waveforms for {${\rm PSL}(2,{\bf Z})$}.
\newblock {\em Experiment. Math.}, 1(4):275--305, 1992.

\bibitem[HS10]{holowinskyMassEquidistributionHecke2010a}
R.~Holowinsky and K.~Soundararajan.
\newblock Mass equidistribution for {H}ecke eigenforms.
\newblock {\em Ann. of Math. (2)}, 172(2):1517--1528, 2010.

\bibitem[Hum18]{humphriesEquidistributionShrinkingSets2018}
P.~Humphries.
\newblock Equidistribution in shrinking sets and {$L^4$}-norm bounds for
  automorphic forms.
\newblock {\em Math. Ann.}, 371(3-4):1497--1543, 2018.

\bibitem[Kha14]{khanFourthMomentHolomorphic2014}
R.~Khan.
\newblock On the fourth moment of holomorphic {H}ecke cusp forms.
\newblock {\em Ramanujan J.}, 34(1):83--107, 2014.

\bibitem[LL11]{lauSUMSFOURIERCOEFFICIENTS2011}
Y.~Lau and G.~L\"{u}.
\newblock Sums of {F}ourier coefficients of cusp forms.
\newblock {\em Q. J. Math.}, 62(3):687--716, 2011.

\bibitem[RS06]{rudnickLowerBoundsMoments2006}
Z.~Rudnick and K.~Soundararajan.
\newblock Lower bounds for moments of {$L$}-functions: symplectic and
  orthogonal examples.
\newblock In {\em Multiple {D}irichlet series, automorphic forms, and analytic
  number theory}, volume~75 of {\em Proc. Sympos. Pure Math.}, pages 293--303.
  Amer. Math. Soc., Providence, RI, 2006.

\bibitem[RS15]{radziwillMomentsDistributionCentral2014}
M.~Radziwi{\l\l} and K.~Soundararajan.
\newblock Moments and distribution of central {$L$}-values of quadratic twists
  of elliptic curves.
\newblock {\em Invent. Math.}, 202(3):1029--1068, 2015.

\bibitem[Shi75]{shimuraHolomorphyCertainDirichlet1975}
G.~Shimura.
\newblock On the holomorphy of certain {D}irichlet series.
\newblock {\em Proc. London Math. Soc. (3)}, 31(1):79--98, 1975.

\bibitem[Sou09]{soundararajanMomentsRiemannZeta2009a}
K.~Soundararajan.
\newblock Moments of the {R}iemann zeta function.
\newblock {\em Ann. of Math. (2)}, 170(2):981--993, 2009.

\bibitem[SZ54]{salemPropertiesTrigonometricSeries1989}
R.~Salem and A.~Zygmund.
\newblock Some properties of trigonometric series whose terms have random
  signs.
\newblock {\em Acta Math.}, 91:245--301, 1954.

\bibitem[Wat02]{watsonRankinTripleProducts2008a}
T.~C. Watson.
\newblock {\em Rankin triple products and quantum chaos}.
\newblock ProQuest LLC, Ann Arbor, MI, 2002.
\newblock Thesis (Ph.D.)--Princeton University.

\end{thebibliography}

\end{document}